\documentclass[12pt]{amsart}
\usepackage{amssymb,amsmath,amsthm}
\usepackage{epsfig}
\usepackage{graphicx}

\newtheorem{theorem}{Theorem}[section]
\newtheorem{lemma}[theorem]{Lemma}
\newtheorem{remark}[theorem]{Remark}
\newtheorem{definition}[theorem]{Definition}

\newtheorem{proposition}[theorem]{Proposition}
\newtheorem{corollary}[theorem]{Corollary}
\newtheorem{problem}[theorem]{Problem}
\newtheorem{example}[theorem]{Example}
\newtheorem{construction}[theorem]{Construction}
\newtheorem{observation}[theorem]{Observation}
\long\def\symbolfootnote[#1]#2{\begingroup%
\def\thefootnote{\fnsymbol{footnote}}\footnote[#1]{#2}\endgroup}

\newcommand\Z{{\mathbb Z}}

\begin{document}

\title{Distributivity versus associativity in the homology theory of algebraic structures}
\author{J\'ozef H. Przytycki}

\centerline{Gdansk, July 2010 -- Bethesda, September 2011}

\thispagestyle{empty}

\begin{abstract}
While homology theory of associative structures, such as groups and rings,
has been extensively studied in the past beginning with the work of Hopf, Eilenberg, and
Hochschild, homology of non-associative distributive structures, such as quandles, were neglected until
recently. Distributive structures have been studied for a long time. In 1880, C.S.~Peirce 
 emphasized the importance of (right) self-distributivity in algebraic
structures. However, homology for these universal algebras was introduced only sixteen
years ago by Fenn, Rourke, and Sanderson. We develop this theory in the historical
context and propose a general framework to study homology of distributive structures. 
We illustrate the theory by computing some examples of 1-term and 2-term homology, and then 
discussing 4-term homology for Boolean algebras. 
We outline potential relations to Khovanov homology, via the Yang-Baxter operator.

\end{abstract}

\maketitle

\tableofcontents

\section{Introduction}
This paper is a summary of numerous talks I gave last year\footnote{From my Summer 2010 talk at 
Knots in Poland III to a seminar at Warsaw Technical University in June 2011 (e.g. Knots in Chicago 
conference, AMS meeting in Chile, Conference in San Antonio, Colloquium at  U. Louisiana, seminars 
at GWU, Columbia University, George Mason U., Universidad de Valparaiso, 
SUNY at Buffalo, University of Maryland, Warsaw University and Gdansk University, 
and Knots in Washington XXXI and XXXII).  I am grateful for the opportunity 
given and stress that I gained a lot from interaction with the audience.}. 
My goal was to understand homology 
theory related to distributivity (and motivated by knot theory), but along the way I discovered various 
elementary structures, probably new, or at least not studied before. Thus I will devote the second section 
to the monoid of binary operations and its elementary properties. This, in addition to being a basis 
for my multi-term distributive homology, may be of interest to  people working on universal algebras.

Because I hope for a broad audience I do not assume any specific knowledge of homological algebra 
or algebraic topology and will survey the basic concepts like chain complex chain homotopy and 
abstract simplicial complex in Sections 3 and 4. In the fifth section I recall two classical 
approaches to homology of a semigroup: group homology and Hochschild homology. In the sixth section 
we build a one-term homology of  distributive structures and recall the definition of the rack homology 
of Fenn, Rourke, and Sanderson \cite{FRS}. In further sections we deepen our study of distributive homology, define 
multi-term distributive homology and show a few examples of computations. In the tenth section we 
relate distributivity to the third Reidemeister move (or braid relation) and discuss motivation coming from 
knot theory. In a concluding remark we speculate on relations with the Yang-Baxter operator and 
a potential path to Khovanov homology.

\markboth{\hfil{\sc Distributivity versus associativity }\hfil}
{\hfil{\sc Jozef H. Przytycki}\hfil}

\section{Monoid of binary operations}

Let $X$ be a set and $*:X\times X \to X$ a binary operation. We call $(X;*)$ a {\it magma}. 
For any $b\in X$ the adjoint maps $*_b: X\to X$, is defined by $*_b(a)=a*b$.
Let $Bin(X)$ be the set of all binary operations on $X$.
\begin{proposition}\label{Proposition 2.1} $Bin(X)$ has a monoidal 
(i.e. semigroup with identity) structure with  composition 
$*_1*_2$ given by $a*_1*_2b= (a*_1b)*_2b$, and the identity $*_0$ being the right trivial operation, that is,  
$a*_0b=a$ for any $a,b\in X$.
\end{proposition}
\begin{proof} Associativity follows from the fact that adjoint maps $*_b$ 
compose in an associative way, $(*_3)_b((*_2)_b(*_1)_b) = ((*_3)_b(*_2)_b)(*_1)_b$; we can write 
directly: $a(*_1*_2)*_3b= ((a*_1b)*_2b)*_3b = (a*_1b)(*_2*_3)b= a*_1(*_2*_3)b$.  
\end{proof}
The submonoid of $Bin(X)$ of all invertible elements in $Bin(X)$ is a group denoted by $Bin_{inv}(X)$.
If $* \in Bin_{inv}(X)$ then $*^{-1}$ is usually denoted by $\bar *$.

It is worth mentioning here that the composition of operations in the monoid $Bin(X)$ may be thought as 
taking first the diagonal coproduct $\Delta: X\to X\times X$ (i.e., $\Delta(b) = (b,b)$) and 
applying the result on $a\in X$; Berfriend Fauser suggested after my March talk in San Antonio 
to try other comultiplications  (he did some unpublished work on it).

One should also remark that $*_0$ is distributive with respect to any other operation, that is,  
$(a*b)*_0c= a*b= (a*_0c)*(b*_0c)$, and $(a*_0b)*c= a*c= (a*c)*_0(b*c)$. 
This distributivity later plays an important role\footnote{Notice that $*_0$ and $*$ are seldom associative,
as $(a*_0b)*c=a*c$ but $a*_0(b*c)=a$.}.

While the associative magma has been called a semigroup for a long time,  the right self-distributive magma
didn't have an established  name, even though C.S.~Peirce considered it in 1880. Alissa Crans, in her PhD thesis of 2004,
suggested the name  {\it right shelf} (or simply  {\it shelf}) \cite{Cr}. Below we write the formal definition 
of a shelf and the related notions of {\it spindle}, {\it rack}, and {\it quandle}.

\begin{definition}\label{Definition 2.2} Let $(X;*)$ be a magma, then: 
\begin{enumerate}
\item[(i)] If $*$ is right self-distributive, that is, $(a*b)*c=(a*c)*(b*c)$,
then $(X;*)$ is called a shelf.
\item[(ii)] If a shelf $(X;*)$ satisfies the idempotency condition, $a*a=a$ for any $a\in X$, then it 
is called a {\it right spindle}, or just a spindle (again the term coined by Crans).
\item[(iii)] If a shelf $(X;*)$ has $*$ invertible in $Bin(X)$ (equivalently $*_b$ is a bijection for any $b\in X$),
then it is called a {\it rack} (the term wrack, like in ``wrack and ruin", of J.H.Conway from 1959 \cite{C-W}, 
was modified to rack in \cite{F-R}).
\item[(iv)] If a rack $(X;*)$ satisfies the idempotency condition, then it is called a {\it quandle} (the term 
coined in Joyce's PhD thesis of 1979; see \cite{Joy}). Axioms of a quandle were motivated by three 
Reidemeister moves (idempotency by the first move, invertibility by the second, and right self-distributivity 
by the third move); see Section 10 and Figures 10.2-10.4.
\item[(v)] If a quandle $(X;*)$ satisfies $**=*_0$ (i.e. $(a*b)*b=a$) then it is called  {\it kei} or 
an involutive quandle. The term kei (\psfig{figure=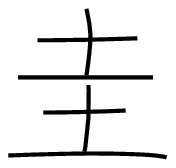,height=0.3cm}) 
was coined in a pioneering paper by Mituhisa Takasaki in 1942 \cite{Tak}
\end{enumerate}
\end{definition}
The main early example of a rack (and a quandle) was a group $G$ with a $*$ operation given 
by conjugation, that is, $a*b=b^{-1}ab$; Conway jokingly thought about it as a wrack of a group.
 The premiere example given by Takasaki was to take an abelian group and define $a*b=2b-a$.
We will give many more examples later (mostly interested in the possibility of having shelves which are
not racks; e.g. Definition \ref{Definition 2.13}).

Definition \ref{Definition 2.2} describes properties of an individual magma $(X;*)$. It is also useful to 
consider subsets or submonoids of $Bin(X)$ satisfying the related conditions described in Definition 
\ref{Definition 2.3}.
\begin{definition}\label{Definition 2.3}
\begin{enumerate}
\item[(1)] 
We say that a subset $S \subset Bin(X)$ is a distributive set if all
pairs of elements $*_{\alpha},*_{\beta} \in S$ are right distributive,
 that is, $ (a*_{\alpha}b)*_{\beta}c= (a*_{\beta}c)*_{\alpha}(b*_{\beta}c)$ (we allow  $*_{\alpha}=*_{\beta}$).
\begin{enumerate}
\item[(i)] 
The pair $(X;S)$ is called a multi-shelf if $S$ is a distributive set. 
If $S$ is additionally a submonoid (resp. subgroup) of $Bin(X)$, we say that it is a distributive monoid (resp. group).
\item[(ii)] 
If $S \subset Bin(X)$ is a distributive set such that each $*$ in $S$ satisfies the idempotency condition, we call 
$(X;S)$ a multi-spindle. 
\item[(iii)] We say that $(X;S)$ is a multi-rack if $S$ is a distributive set, and  all elements of $S$ are invertible. 
\item[(iv)] We say that $(X;S)$ is a multi-quandle if $S$ is a distributive set, and 
elements of $S$ are invertible and satisfy the idempotency condition.
\item[(v)] We say that $(X;S)$ is a multi-kei if it is a multi-quandle with $**=*_0$ for any $*\in S$. 
Notice that if $*_1^2=*_0$ and $*_2^2=*_0$ then $(*_1*_2)^2=*_0$; more generally if 
$*_1^n=*_0$ and $*_2^n=*_0$ then $(*_1*_2)^n=*_0$. This follows from Proposition \ref{Proposition 2.8}.
\end{enumerate}
\item[(2)] We say that a subset $S \subset Bin(X)$ is an associative set if all     
pairs of elements $*_{\alpha},*_{\beta} \in S$ are associative with respect to each another, that is, 
$ (a*_{\alpha}b)*_{\beta}c= a*_{\alpha}(b*_{\beta}c)$.
\end{enumerate}
\end{definition}

\begin{proposition}\label{Proposition 2.4}
\begin{enumerate}
\item[(i)] If $S$ is a distributive set and $*\in S$ is invertible, then $S\cup \{\bar *\}$ is also a 
distributive set.
\item[(ii)] If $S$  is a distributive set and $M(S)$ is the monoid generated by $S$ then $M(S)$ is a 
distributive monoid.
\item[(iii)] If $S$  is a distributive set of invertible operations and $G(S)$ is the group generated by $S$, then 
$G(S)$ is a distributive group.
\end{enumerate}
\end{proposition}

We divide our proof into three elementary but important lemmas.
\begin{lemma}\label{Lemma 2.5} Let $(X;*)$ be a magma and $f:X\to X$ a magma homomorphism (i.e. $f(x*y)=f(x)*f(y)$).
If $f$ is invertible (we denote $f^{-1}$ by $\bar f$) then $\bar f$ is also a magma homomorphism.
\end{lemma}
\begin{proof}
Our goal is to show that $\bar f(x*y)=\bar f(x)*\bar f(y)$. For this, let $\bar x = \bar f(x)$ and 
$\bar y = \bar f(y)$ (equivalently $f(\bar x)=x$ and $f(\bar y)=y$). Then, from $f(\bar x*\bar y)=f(\bar x)*f(\bar y)$ 
follows $f(\bar x*\bar y)= x*y$. Therefore, $\bar x*\bar y = \bar f(x*y)$ which gives $\bar f(x)*\bar f(y)=\bar f(x*y)$. 
\end{proof}

\begin{corollary}\label{Corollary 2.6} 
\begin{enumerate}
\item[(i)] If $*,*' \in Bin(X)$ and $*$ is invertible and (right) distributive with respect to $*'$, then 
$\bar *$ is (right) distributive with respect to $*'$.
\item[(ii)]  If $*,*' \in Bin(X)$, $*$ is invertible, and $*'$ is (right) distributive with respect to $*$, then 
$*'$ is (right) distributive with respect to $\bar *$.
\item[(iii)] If $(X;*)$ is a rack, then $(X;\bar *)$ is a rack.
\item[(iv)]
If $\{*',*\}$ is a distributive set and $*$ is invertible, then $\{*',*,\bar *\}$ is a distributive set.
\end{enumerate}
\end{corollary}

\begin{proof}
(i) Because $(a*'b)*c= (a*c)*'(b*c)$, the map $*_c : X\to X $ is a $*'$-shelf homomorphism; thus
by Lemma \ref{Lemma 2.5}, $\bar *_c: X\to X $ is a $*'$-shelf homomorphism. The last property can be written 
as $(a*'b)\bar *c= (a \bar *c)*'(b \bar *c)$, that is, $\bar *$ is (right) distributive with respect to 
$*'$.\\
(ii)
 To prove the distributivity of $*'$ with respect to $\bar *$ we consider the formula that follows 
from the distributivity of $*'$ with respect to $*$:
$ ((a\bar * b)*b)*'c = ((a\bar * b)*'c)*(b*'c)$. This is equivalent to 
$ a*'c = ((a\bar * b)*'c)*(b*'c)$ and thus:
$$ (a*'c)\bar * (b*'c)= ((a\bar * b)*'c).$$
(iii) To see the (right) self-distributivity of $\bar *$ we notice that the (right) self-distributivity of 
$*$ gives, by (ii), the distributivity of $*$ with respect to $\bar *$. Thus, 
 $*_c$ is a $\bar *$-shelf homomorphism so, by Lemma \ref{Lemma 2.5}, $\bar *_c$ is a $\bar *$-shelf homomorphism 
which gives the (right) self-distributivity of $\bar *$.\\
(iv) follows from (i), (ii), and (iii).
\end{proof}
Proposition \ref{Proposition 2.4}(i) follows from Corollary \ref{Corollary 2.6}. Part (ii) of 
Proposition \ref{Proposition 2.4} follows from the following elementary lemma, and (iii) is a combination of (i) and (ii).

\begin{lemma}\label{Lemma 2.7}
\begin{enumerate}
\item[(i)] Let $*,*_1, *_2\in Bin(X)$ and let $*$ be (right) distributive with respect to $*_1$ and $*_2$. Then 
$*$ is (right) distributive with respect to $*_1*_2$.
\item[(ii)]  Let $*,*_1, *_2\in Bin(X)$ and let $*_1$  and $*_2$ be  (right) distributive with respect to $*$. Then 
 $*_1*_2$ is (right) distributive with respect to $*$.
\item[(iii)]
If $\{S,*_1,*_2\}$ is a distributive set, then $\{S,*_1,*_2,*_1*_2\}$ is also a distributive set.
\end{enumerate}
\end{lemma}
\begin{proof} (i) We have $(a*_1*_2b)*c= ((a*_1b)*_2b)*c=((a*c)*_1(b*c))*_2(b*c)=(a*c)*_1*_2((b*c)$, as needed.\\
(ii) We have $(a*b)*_1*_2c= ((a*b)*_1c)*_2c= ((a*_1c)*_2c)*((b*_1c)*_2c)= (a*_1*_2c)*(b*_1*_2c)$, as needed.\\
(iii) Because of (i) and (ii) we have to only prove the (right) self-distributivity of $*_1*_2$. 
We have $$(a*_1*_2b)*_1*_2c=(((a*_1b)*_2b)*_1c)*_2c= (((a*_1b)*_1c)*_2(b*_1c))*_2c=$$
 $$(((a*_1c)*_1(b*_1c))*_2c)*_2((b*_1c)*_2c)=$$
$$ ((a*_1c)*_2c)*_1((b*_1c)*_2c))*_2((b*_1c)*_2c)) = (a*_1*_2c)*_1*_2(b*_1*_2c).$$
This proves the (right) self-distributivity of $*_1*_2$.
\end{proof}

Our monoidal structure of $Bin(X)$ behaves well with respect to (right) distributivity, as 
demonstrated by Proposition \ref{Proposition 2.4}. It is interesting to notice that the analogue of 
Proposition \ref{Proposition 2.4} does not hold for associative sets. For example, if $(X;*)$ is a group,
then, $\bar *$ is seldom associative. 
Similarly, it very seldom happens that if $\{*_1,*_2\}$ is an associative set then the operation $*_1*_2$ 
is associative. 
\\ \ \\
\subsection{When is a distributive monoid commutative?}

Soon after I gave the definition of a distributive submonoid of $Bin(X)$ Michal Jablonowski, a graduate student 
at Gda\'nsk University, noticed that any distributive monoid whose elements are idempotent operations
is commutative. We have:
\begin{proposition}\label{Proposition 2.8}
\begin{enumerate}
\item[(i)]
 Consider $*_{\alpha},*_{\beta}\in Bin(X)$ such that $*_{\beta}$ is idempotent ($a*_{\beta}a=a$) and
distributive with respect to $*_{\alpha}$, then $*_{\alpha}$ and $*_{\beta}$ commute. In particular:
\item[(ii)] If $M$ is a distributive monoid and $*_{\beta}\in M$ is an idempotent operation, then $*_{\beta}$ 
is in the center of $M$.
\item[(iii)] A distributive monoid whose elements are idempotent operations is commutative.
\end{enumerate}
\end{proposition}
\begin{proof} We have: $(a*_{\alpha}b)*_{\beta}b \stackrel{distrib}{=} (a*_{\beta}b)*_{\alpha}(b*_{\beta}b)
\stackrel{idemp}{=}
(a*_{\beta}b)*_{\alpha}b$.
\end{proof}
A few months later Agata Jastrz{\c e}bska (also  a graduate student at Gda\'nsk University), 
checked that any distributive group in $Bin_{inv}(X)$ for  $|X|\leq 5$  is 
commutative. Finally, in July of 2011 Maciej Mroczkowski (attending my series of talks at Gda\'nsk University) 
constructed noncommutative distributive submonoids of $Bin(X)$, the smallest for $|X|=3$. Here is Mroczkowski's 
construction.
\begin{construction}\label{Construction 2.9}
Consider a pair of sets $X\supset A$ and the set of all retractions from $X$ to $A$ (denoted by
$R(X,A)$. Then the set of all shelfs $(X;*_r)$ with $r\in R(X,A)$ and $a*_rb=r(b)$ forms a distributive subsemigroup
of $Bin(X)$ which is non-abelian for $|X|> |A|>1$. This semigroup, denoted $SR(X,A)$, has a  presentation:
$\{R\ | \ *_{r_{\alpha}}*_{r_{\beta}}= *_{ r_{\beta}}\}$ and is clearly not commutative.
 Notice that it is a semigroup with a left trivial operation.

The simplest example is given by $X=\{b,a_1,a_2\}$ and $A=\{a_1,a_2\}$; then $SR(X,A)$ has 2 elements
$*_{r_1}$ and $*_{r_2}$ with $r_1(b)=a_1$ and $r_2(b)=a_2$.

The choices above are related to the following:\\
(i) $(X;*_g)$ with $a*_gb=g(b)$ is a shelf if and only if $g^2=g$.\\
(ii) Two operations $*_{g_1}$ and $*_{g_2}$ are distributive with respect to each other iff 
$g_1g_2=g_2$ and $g_2g_1=g_1$, since:\\
$(a*_{g_1}b)*_{g_2}c=g_2(c)$ and $(a*_{g_2}c)*_{g_1}(b*_{g_2}c)=g_1(b*_{g_2}c)=g_1(g_2(c))$. \\
(iii) $g_1$ and $g_2$ form a distributive set if $g_1(X)=g_2(X)$ and $g_1$ and $g_2$ are retractions.

$SR(X,A)$ is a distributive semigroup.  If we add $*_0$ to it we obtain a distributive monoid $MR(X,A)$.
\end{construction}

It still remains an open problem whether an invertible operation is in the center 
of a distributive submonoid of $Bin(X)$,  or whether a distributive subgroup of $Bin(X)$ 
is abelian. With relation to these questions, we 
propose a few problems for a computer savvy student, possibly for her/his senior thesis or master degree:
\begin{problem}
\begin{enumerate}
\item[(i)] For small $X$, say $|X| \leq 6$, find all distributive submonoids of $Bin(X)$.
In fact, such monoids form a poset with respect to inclusion, so it is sufficient to find all
maximal distributive monoids.
\item[(ii)]  Consider only distributive subgroups of $Bin(X)$.
As in (i) find all maximal subgroups. Are they all abelian?
\item[(iii)]  Now assume that we have a distributive monoid of idempotent operations 
(not necessarily invertible). Again find maximal distributive monoids in this category.
It is interesting that, for $|X| =2$, we have four-spindle structures and they form
a unique maximal distributive submonoid of 4 elements (related to the two element Boolean algebra).
\item[(iv)] Consider now submonoids of $Bin(X)$ such that their elements satisfy all quandle conditions.
Find all maximal distributive subgroups of  $Bin(X)$  in this category. 
This is stronger than classifying small quandles since we build posets of them.
\end{enumerate}
\end{problem}
For $|X| = 6$ the problems above may test the strength of a computer and the quality of the algorithm.
For $|X| = 5$ it is feasible and
for $|X| = 4$ even a small computer and not that efficient program should work and a solution
will still be of great interest.

\subsection{Every abelian group is a distributive subgroup of $Bin(X)$ for some $X$}\ 
 
In the previous subsection we stressed that the question of whether every distributive 
subgroup of $Bin(X)$ is abelian is open; it is easy, however, to construct any abelian group 
 as a distributive subsemigroup of some $Bin(X)$.  The following proposition describes 
an elementary generalization of this: 

\begin{proposition}\label{Proposition 2.11} 
Let $X$ be a semigroup. Consider a map $\tau : X \to Bin(X)$ given by $x \tau(a)y = xa$. Then:
\begin{enumerate} 
\item[(i)] $\tau$ is a homomorphism of semigroups.
\item[(ii)] If $1_r$ is a right unit of $X$ (i.e. $x1_r=x$) then $\tau (1_r) = *_0$.
\item[(iii)] If $X$ is a group, or more generally a semigroup with the property\footnote{Functions 
satisfying this property are called functionally equal. The property 
holds, for example, for for a semigroup with the left cancellation property $xa=xb \Rightarrow a=b$,
or an abelian semigroup whose elements are all idempotent (if $xa=xb$ for every $x$, then $a=aa=ab=ba=bb=b$).} 
that if $xa = xb$
for every $x\in X$ then $a=b$, 
then $\tau$ is a monomorphism. 
\item[(iv)] For any function $f:X\to X$ we define a shelf $(X;*_f)$ by $a*_fb = f(a)$ (this is a
rack if $f$ is invertible and a spindle if $f=Id_X$). Then $\{*_{f_1},*_{f_2}\}$ forms a
distributive set iff $f_1$ and $f_2$ commute.
\item[(v)] If  $X$ is a commutative semigroup  such that if $xa = xb$ for any $x$ then $a=b$, then $X$ embeds as 
a distributive subsemigroup in $Bin(X)$.
\end{enumerate}
\end{proposition}
\begin{proof} (i) We have $x\tau(ab)y= xab$ and $x\tau(a)\tau(b)y=(x\tau(a)y)\tau(b)y= xa\tau(b)y=xab$.\\
(ii) $x\tau(1_r)y=x1_r=x$ thus $\tau(1_r)=*_0$.\\ 
(iii) If $\tau(a)= \tau(b)$, then for all $x$ we have $xa=xb$. Thus, by our property, $a=b$ and $\tau$ is a monomorphism.\\
(iv) We have: $(a*_{f_1}b)*_{f_2}c= f_2(a*_{f_1}b)=f_2f_1(a)$,\\
$(a*_{f_2}c)*_{f_1}(b*_{f_2}c)= f_1((a*_{f_2}c)=f_1f_2(a)$. Thus, right distributivity holds 
iff  $f_1$ and $f_2$ commute. \\
(v) With our assumption $\tau$ is a monomorphism, and by (iv) its image is a distributive semigroup
 (compare Proposition \ref{Proposition 7.2} where we show that commutativity of $X$ is not needed if we 
replace distributivity by chronological-distributivity).
\end{proof}

\subsection{Multi-shelf homomorphism}

Homomorphism of multi-shelves is a special case of a homomorphism of universal algebras (heterogeneous 
two-sorted algebras). Concretely,
consider two multi-shelves $(X_1;S_1)$ and $(X_2;S_2)$ and a map $h: S_1 \to S_2$. 
We say that $f: X_1 \to X_2$ is a multi-shelves homomorphism if for any $* \in S_1$ we have 
$f(a*b)= f(a)h(*)f(b)$. 
\begin{proposition}\label{Proposition 2.12} 
Let $(X;S)$ be a multi-shelf and $* \in S$. Then for any $c\in X$ the adjoint map $*_c: X \to X$ is a multi-shelf 
endomorphism of $X$ (with $h=Id: S \to S$).
\end{proposition}
\begin{proof}
The map is a homomorphism because, for any $*_{\alpha}\in S$, from  right distributivity we have:\\
$*_c(a*_{\alpha}b)=(a*_{\alpha}b)*c= (a*c)*_{\alpha}(b*c)= *_c(a)*_{\alpha}*_c(b).$ 
\end{proof}

\subsection{Examples of shelves and multi-shelves from a group}\

Consider the three classical classes of quandles:\ 
the trivial quandles, the conjugate quandles, and the core quandles. 
They have (also classical) generalizations (e.g. \cite{Joy,A-G}), or we can say 
deformations, important for us because they also produce interesting shelves 
which are often not quandles or racks, and
lead to interesting families of multi-shelves.
\begin{definition}\label{Definition 2.13} Let $G$ be a group and $h: G \to G$ a group homomorphism.
Then we define three classes of spindles with $(G,*_h)$ as follows:\\
     (i) $a*_hb = h(ab^{-1})b$;\\
    (ii) $a*_hb =  h(b^{-1}a)b$;\\
   (iii) $a*_hb = h(ba^{-1})b$, here we assume that $h^2 = h$.
\end{definition}
   We comment on each class below:\\
 (i) $(G,*_h)$ is a quandle iff  $h$ is invertible, and for $h = Id$ it is a trivial quandle.
If $G$ is an abelian group we obtain an Alexander spindle (Alexander quandle for $h$ invertible);
in an additive convention we write $a *_h b = h(a)-h(b) + b = (1- h)(b) + h(a)$.\\
(ii) $(G,*_h)$ is a quandle iff  $h$ is invertible, and for $h = Id$ we obtain the conjugacy
quandle ($a*b=b^{-1}ab$). If $G$ is an abelian group we obtain an Alexander spindle, the same as in case (i).\\
(iii) We need $h^2 = h$ for right self-distributivity, as the following calculation demonstrates:
 $$(a*_hb)*_hc=(h(ba^{-1})b)*_hc = h(c(h(ba^{-1})b)^{-1})c= h(cb^{-1})h^2(ab^{-1})c \stackrel{h^2=h}{=}$$
$$ h(cb^{-1} ab^{-1} )c$$
$$(a*_h c)*_h (b*_h c) = h((b*_h c)(a*_h c)^{-1})(b*_h c)=h(h(cb^{-1})c(h(ca^{-1} )c)^{-1}))h(cb^{-1})c {=}$$
     $$  h^2(cb^{-1} )h(c)h(c^{-1} )h^2 (ac^{-1} )h(cb^{-1} )c = h^2 (cb^{-1} ac^{-1} )h(cb^{-1} )c \stackrel{h^2=h}{=}$$
                                      $$h(cb^{-1} ab^{-1} )c.$$
    Because of the condition $h^2 = h$, our spindle is a quandle only if $h = Id$,
in which case we obtain a core quandle ($a*b = ba^{-1} b$).\\
It is interesting to compose $*_h*_h$ in (iii), as we obtain example (i). We
can interpret this by saying that $*_h$ from (i), for $h^2 = h$ has a square root. One
can also check that for (iii) $*_h^3 = *_h$, thus the monoid in $Bin(X)$ generated by $*_h$
is the three element cyclic monoid $\{*_h |\  *_h^3 = *_h \}$.
We have: $a*_h^3b= ((a*_hb)*_hb)*_hb= ((h(ba^{-1})b)*_hb)*_hb=(h(bb^{-1}h(ab^{-1})b))*_hb =
(h^2(ab^{-1})b)*_hb \stackrel{h^2=h}{=} (h(ab^{-1})b)*_hb= h(bb^{-1}h(ba^{-1})b=h^2(ba^{-1})b
\stackrel{h^2=h}{=} h(ba^{-1})b= a*_hb.$ 
\\
    Let us go back to case (ii):\\
 We check below that $*_h$ given by $a *_h b = h(b^{-1} a)b$ is right self-distributive. 
Thus by Proposition \ref{Proposition 2.4}(ii) the monoid generated by $*_h$ is a distributive monoid; 
however $*_{h_1}$ and $*_{h_2}$ are seldom right distributive as the calculation below shows (proving 
also distributivity for $h_1=h_2$):
           $$ (a *_{h_1} b) *_{h_2} c = (h_1 (b^{-1} a)b) *_{h_2} c = h_2 (c^{-1} (h_1 (b^{-1} a)b)c =$$
                              $$ h_2 (c^{-1} )h_2 h_1 (b^{-1} a)h_2 (b)c$$
               $$(a *_{h_2} c) *_{h_1} (b *_{h_2} c) = (h_2 (c^{-1} a)c) *_{h_1} (h_2 (c^{-1} b)c) =$$
$$h_1 (h_2 (c^{-1} b)c)^{-1} h_2 (c^{-1} a)c)h_2 (c^{-1} b)c =
h_1 (c^{-1} )h_1 h_2 (b^{-1} c)h_{1} h_2 (c^{-1} a)h_1 (c)h_2 (c^{-1} b)c =$$
                        $$  h_1 (c^{-1} )h_1 h_2 (b^{-1} a)h_1 (c)h_2 (c^{-1})h_2( b)c.$$
 Again back in case (i) ($a *_h b = h(ab^{-1} )b$) we get:
                     $$(a*_{h_1}b)*_{h_2}c = (a *_{h_2} c) *_{h_2 h_1 h_2^{-1}} (b *_{h_2} c).$$
In particular, $*_{h_1}$ and $*_{h_2}$ are right distributive if the functions commute
($h_1 h_2 = h_2 h_1$).
The last equation can be interpreted as twisted distributivity, for $G$-families of quandles,  
the concept developed by Ishii, Iwakiri, Jang, and Oshiro \cite{Is-Iw,Ca-Sa,IIJO}.

In the next few sections we  compare associativity and distributivity in developing homology theory.
In Section 3 we recall the basic notions of a chain complex, homology, and a chain homotopy, in 
order to make this paper accessible to non-topologists. We also recall the notion of a presimplicial and 
simplicial module, the basic concepts that are not familiar to nonspecialists.

\section{Chain complex, 
homology, and chain homotopy}

Let $\{C_n\}_{n\in Z}$ be a graded abelian group (or an $R$-module\footnote{For simplicity we work mostly 
with abelian groups, i.e. $\Z$-modules, but we could also assume that we work with $R$-modules, where 
$R$ is a ring with identity.}). A chain complex ${\mathcal C}= \{C_n,\partial_n\}$
is a sequence of homomorphisms $\partial_n: C_n \to C_{n-1}$ such that
$\partial_{n-1}\partial_{n}=0$ for any $n$. 
So $\mathrm{Im}(\partial_{n+1}) \subset \mathrm{Ker}(\partial_{n})$, and the quotient group
$\frac{\mathrm{Ker}(\partial_{n})}{\mathrm{Im}(\partial_{n+1})}$ is called the $n$th homology of a chain
complex ${\mathcal C}$, and denoted by $H_n({\mathcal C})$. Elements of $\mathrm{Ker}(\partial_{n})$ are called
n-cycles, and we write $Z_n=\mathrm{Ker}(\partial_{n})$, and elements of $\mathrm{Im}(\partial_{n+1})$ are called
n-boundaries and we write $B_n=\mathrm{Im}(\partial_{n+1})$ .

A map of chain complexes $f: {\mathcal C}' \to {\mathcal C}$ is a collection of group homomorphisms 
$f_n: C'_n \to C_n$ such that all squares in the diagram commute, that is, $ f_{n-1}\partial'_n= \partial_{n}f_n$.
A chain map induces a map on homology $f_*: H_n( {\mathcal C}') \to H_n( {\mathcal C})$.

One important and elementary tool we use in the paper is a {\it chain homotopy}, so we recall the notion: 
\begin{definition}\label{Definition 3.1}
Two chain maps $f,g : {\mathcal C}' \to {\mathcal C}$
are chain homotopic if there is a degree 1 map $h: C' \to C$
(that is $h_i: C'_i \to C_{i+1}$) such that
$$f-g = {\partial_{i+1}}h_i + h_{i-1}\partial'_i.$$
\end{definition}
The importance of chain homotopy is given by the following classical result:
\begin{theorem} If two chain maps $f$ and $g$ are chain homotopic, 
then they induce the same homomorphism of homology
$f_* =g_*: H({\mathcal C}') \to H({\mathcal C})$.
In particular, if ${\mathcal C}'= {\mathcal C}$, $f=Id$, and $g$ is the zero map, then  
the chain complex ${\mathcal C}$ is acyclic, that is $H_n({\mathcal C})=0$ for any $n$.
\end{theorem}

\subsection{Presimplicial module and Simplicial module}
It is convenient to have the following terminology,  whose usefulness is visible in the next sections and 
which takes into account the fact that, in most homology theories, the boundary operation
$\partial_n:C_n \to C_{n-1}$ can be decomposed as an alternating sum of {\it face maps} $d_i:C_n \to C_{n-1}$.
Often we also have {\it degeneracy} maps $s_i:C_n \to C_{n+1}$. Formal definitions mostly follow \cite{Lod}.

\begin{definition}\label{Definition 3.3}\  
\begin{enumerate}
\item[(Sim)]
A simplicial module $(C_n,d_i,s_i)$, over a ring $R$, is a collection of
$R$-modules $C_n$, $n\geq 0$, together with face maps $d_i:C_n\to C_{n-1}$ and degenerate maps
$s_i: C_n\to C_{n+1}$, $0\leq i \leq n$, which satisfy the following properties:
$$ (1) \ \ \  d_id_j = d_{j-1}d_i\ for\ i<j. $$
$$(2)\ \ \ s_is_j=s_{j+1}s_i,\ \ 0\leq i \leq j \leq n, $$
$$ (3) \ \ \ d_is_j= \left\{ \begin{array}{rl}
 s_{j-1}d_i &\mbox{ if $i<j$} \\
s_{j}d_{i-1} &\mbox{ if $i>j+1$}
       \end{array} \right.
$$
$$ (4) \ \ \ d_is_i=d_{i+1}s_i= Id_{C_n}. $$

\item[(Presim)] 
$(C_n,d_i)$ satisfying (1) is called a {\it presimplicial module} and leads to the chain complex
$(C_n,\partial_n)$ with $\partial_n = \sum_{i=0}^n(-1)^id_i$.

\item[(W)] A {\it weak simplicial module} $(M_n,d_i,s_i)$ satisfies conditions (1)-(3) and a weaker condition
 in place of condition (4):\\
(4') \ \ \ $d_is_i=d_{i+1}s_i$.
\item[(VW)] A {\it very weak simplicial module} $(M_n,d_i,s_i)$ satisfies conditions (1)-(3). 
\end{enumerate}
\end{definition}

We defined weak and very weak simplicial modules motivated by homology of distributive structures 
(as it will be clear later, Proposition \ref{Proposition 6.4}).
We use the terms weak and very weak simplicial modules as the terms pseudo and almost simplicial modules are 
already in use\footnote{According to \cite{Fra}, a pseudo-simplicial module $(M_n,d_i,s_i)$ satisfies only
conditions (1),(3),(4) of Definition \ref{Definition 3.3} \cite{Ti-Vo,In}. An  almost-simplicial module satisfies
conditions (1)-(4) of Definition \ref{Definition 3.3} except $s_is_i= s_{i+1}s_i$. A pseudo-simplicial
module satisfies the Eilenberg-Zilber Theorem described in \cite{Fra1} and proved in \cite{In}.}.

\subsection{Subcomplex of degenerate elements}

Consider a graded module $(C_n,s_i)$ where $s_i: C_n\to C_{n+1}$ for $0\leq i \leq n$. We define a graded 
module of {\it degenerated submodules} $C_n^D$ as follows:
$$C_n^D= span\{s_0(C_{n-1}),...,s_{n-1}(C_{n-1}) \}.$$  
If $(C_n,d_i,s_i)$ is a presimplicial module with degeneracy maps, 
then  $C_n^D$ forms a subchain complex of $(C_n,\partial_n)$ with $\partial_n=\sum_{i=1}^n(-1)^id_i$ 
provided that conditions (3) and (4') of Definition \ref{Definition 3.3} hold (in particular, if 
$(C_n,d_i,s_i)$ is a weak simplicial module). We compute:
$$\partial_ns_p= (\sum_{i=0}^n(-1)^id_i)s_p= \sum_{i=0}^n(-1)^i(d_is_p)= $$
 $$ \sum_{i=0}^{p-1}(-1)^i(d_is_p) + (-1)^pd_ps_p + (-1)^{p+1}d_{p+1}s_p + \sum_{i=p+2}^{p-1}(-1)^i(d_is_p)=$$
$$\sum_{i=0}^{p-1}(-1)^i(s_{p-1}d_i) + \sum_{i=p+2}^{p-1}(-1)^i(s_pd_{i-1}) \in C^D_{n-1}.$$

It is a classical result that if $(C_n,d_i,s_i)$ is a simplicial module, then $C_n^D$ is an acyclic 
subchain complex. The result does not hold, however, for a weak simplicial module, and we 
can have nontrivial degenerate homology $H_n^D=H_n(C^D)$ and normalized homology $H^{Norm}_n=H_n(C/C^D)$ 
different from $H_n(C)$.
These play an important role in the theory of distributive homology.

\begin{remark}\label{Remark 3.4} Even if $(C_n,d_i,s_i)$ is only a very weak simplicial module, that is $d_is_i$ is 
not necessarily equal to $d_{i+1}s_i$, we can construct the analogue of a degenerate subcomplex. 
We define $t_i: C_n \to C_n$ by $t_i=d_is_i-d_{i+1}s_i$, and define subgroups $C_n^{(t)}\subset C_n$ as 
$span (t_0(C_n),...,t_{n-1}(C_n),t_n(C_n))$.  Then we define the subgroups $C_n^{(tD)}$ as
$span(C_n^{(t)},C_n^D)$. We check directly that $C_n^{(t)}$ and $C_n^{(tD)}$ 
are subchain complexes of $(C_n,\partial_n)$ and they play an important role in distributive homology.
In Theorem \ref{Theorem 6.6} we show how to use the triplet of chain complexes $C_n^{(t)}\subset C_n^{(tD)} \subset C_n$ 
to find the homology of a shelf $(X;*_g)$ with $a*_gb= g(b)$, $g:X \to X$, and $g^2=g$. The generalization 
of this is given in \cite{P-S}.
\end{remark}

\section{Homology for a simplicial complex}

The homology theories that we introduce are modelled on the classical homology of simplicial complexes.
We review this for completeness below.

Let $K=(X,S)$ be an abstract simplicial complex with vertices $X$ (which we order) and simplexes $S\subset 2^X$.
That is, we assume elements of $S$ are finite, include all one-element subsets\footnote{We find
it convenient to also allow an  empty simplex, say of dimension $-1$; it will lead to augmented
chain complexes.}, and that if $s'\subset s \in S$, then also $s'\in S$. 
The associated chain complex has a chain group $C_n$ that is a subgroup of 
$\Z X^{n+1}$ (i.e. a free abelian group with basis $X^{n+1}$) generated by
$n$-dimensional simplexes $(x_0,x_1,...,x_n)$: we assume that $x_0<x_1<...<x_n$ in our
ordering. The boundary operation is defined by:
$$\partial (x_0,x_1,...,x_n) = \sum_{i=0}^n (-1)^i (x_0,...,x_{i-1},x_{i+1},...,x_n).$$
Notice that we can put $d_i(x_0,x_1,...,x_n)= (x_0,...,x_{i-1},x_{i+1},...,x_n)$ 
with $\partial_n=\sum_{i=0}^n (-1)^id_i$,
and that $(C_n,d_i)$ is a simplicial module (i.e. $d_id_j= d_{j-1}d_i$ for $0\leq i <j \leq n$).
 
We do not require any structure on $X$, but as we will see later we can think of $X$ as
a (trivial) semigroup or a shelf, $(X,*_0)$, with \
 $a*_0b=a$ for any $a,b\in X$.

One proves classically that homology does not depend on the ordering of $X$.
Alternatively, one can consider a chain complex with bigger chain groups $\bar C_n\subset \Z X^{n+1}$ 
generated by sequences
$(x_0,x_1,...,x_n)$ such that the set $\{x_0,x_1,...,x_n\}$ is a simplex in $S$; 
as before we put $\partial (x_0,x_1,...,x_n) = \sum_{i=0}^n (-1)^i (x_0,...,x_{i-1},x_{i+1},...,x_n).$
In this approach, our definition is ordering independent and allows degenerated simplexes.
The homology is the same as we can consider the acyclic subcomplex of $\bar C_n$ generated
by degenerate elements $(x_0,x_1,...,x_n)$, that is, elements with $x_i=x_{i+1}$ for some $i$, and
``transposition" elements
$(x_0,...x_{i-1},x_i,x_{i+1},x_{i+2},...,x_n)+ (x_0,...x_{i-1},x_{i+1},x_i,x_{i+2},...,x_n)$. 

In this second approach we have a simplicial module $(C_n,d_i,s_i)$ with $s_i(x_0,...,x_n)=
(x_0,...,x_{i-1},x_i,x_i,x_{i+1},...,x_n)$.

{\bf The motivation for the boundary operation} comes from the geometrical realization of an abstract
simplicial complex as illustrated below:\

$$\partial(x_0,x_1,x_2)= \partial({\parbox{2.7cm}{\psfig{figure=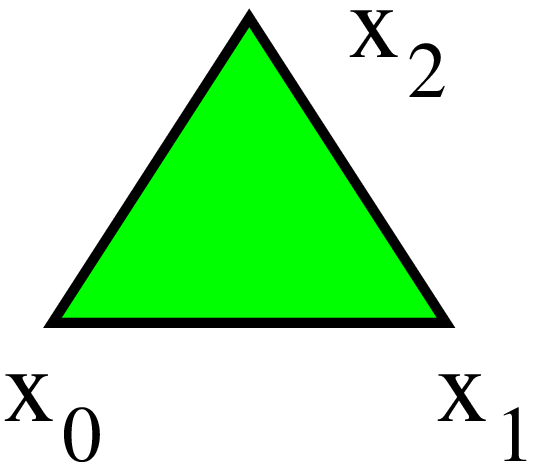,height=2.5cm}}}) =
{\parbox{2.7cm}{\psfig{figure=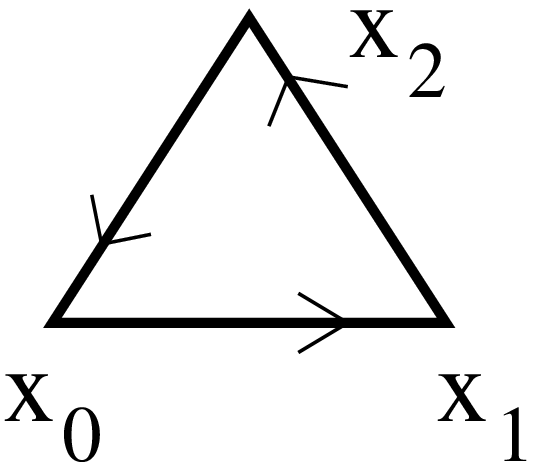,height=2.5cm}}} = $$
$$  (x_1,x_2)-(x_0,x_2)+(x_0,x_1).$$

\section{Homology of an associative structure: group homology and Hochschild homology}

We describe below two classical homology theories for semigroups. Our homology of distributive 
structures is related to these theories.

\subsection{Group homology of a semigroup}
Let $(X,*)$ be a semigroup. We define a chain complex $\{C_n,\partial_n\}$ as follows:
$C_n(X)= \Z X^n$ and $\partial_n : \Z X^n \to \Z X^{n-1}$ is defined by
$$ \partial(x_1,...,x_n)= (x_2,...,x_n) +$$
$$\sum_{i=1}^{n-1}(-1)^i(x_1,...,x_{i-1},x_i*x_{i+1},x_{i+2},...,x_n) +$$
$$(-1)^n(x_1,...x_{n-1})$$
We also assume that $H_0(X)=\Z$ and $\partial_1(x)=1$.
We can check that $ \partial^2=0$ if and only if $*$ is associative.
\begin{example}\label{Example 5.1}
Checking this is quite illuminative, so we perform it for $n=3$:
$$\partial_2(\partial_3(x_1,x_2,x_3))=\partial_2((x_1,x_2) -(x_0*x_1,x_2)+ (x_0,x_1*x_2)-(x_0,x_1))=$$
$$ x_2 - x_1*x_2 +x_1 $$
$$- x_2 + (x_0*x_1)*x_2 -x_0*x_1$$
$$+ x_1*x_2  - x_0*(x_1*x_2) + x_0$$
$$-x_1 + x_0*x_1 - x_0 =$$
$$ (x_0*x_1)*x_2 -x_0*(x_1*x_2),$$
which is 0 iff $*$ is associative.
\end{example}

Let $\partial^{(\ell)}$ be a boundary map obtained from the group homology boundary operation by dropping
the first term from the sum. Analogously, let $\partial^{(r)}$ be a boundary map obtained 
from the group homology boundary operation by dropping the last term from the sum. It is a classical 
observation that $(C_n,\partial^{(\ell})$ and $(C_n,\partial^{(r)})$ are acyclic for a group (or a monoid).
We show this below in a slightly more general context (used later in the distributive case).
\begin{example}\label{Example 5.2}\
\begin{enumerate}
\item[($\ell)$] Assume that a semigroup $(X,*)$ has a left identity $1_{\ell}$ (i.e. $1_{\ell}x=x$),
then the chain homotopy $$H_{\ell}(x_1,...,x_n)= (1_{\ell},x_1,...,x_n)$$
satisfies:
$$(\partial^{(\ell)} H_{\ell} + H_{\ell}\partial^{(\ell)})(x_1,...,x_n)= Id_X.$$
Thus the identity map is chain homotopic to the zero map, and the related homology groups are trivial.
\item[($r$)] Assume that a semigroup $(X,*)$ has a right identity $1_{r}$ (i.e. $x1_r=x$), then 
the chain homotopy 
$$H_{r}(x_1,...,x_n)= (-1)^{n+1}(x_1,...,x_n,1_r)$$
and we get:
$$(\partial^{(r)} H_{r} + H_{r}\partial^{(r)})(x_1,...,x_n)= Id_X.$$
Thus the identity map is chain homotopic to the zero map, and the related homology groups are trivial.
\end{enumerate}
\end{example}
One of the classical observations in group homology is that if $(X,*)$ is a finite group, then  
the cardinality of $X$, $|X|$, annihilates homology groups. We demonstrate this below in a slightly more general context; we 
use the observation later for distributive homology.
\begin{proposition} Assume that $(X,*)$ is a semigroup which contains a finite right orbit  $A$, 
that is, $A$ is a finite subset of $X$ such that for each $b\in X$, we have $*_b(A)=A$ 
(i.e. $*_b:A \to A$ is a bijection). Then $|A|$ annihilates $H_n(X)$. In particular, if 
$(X,*)$ is a finite group we can take $A=X$. If $(X,*)$ has a left zero\footnote{Sometimes called a left 
 projector.} $p_{\ell }$ (i.e. $p_{\ell}*x=p_{\ell}$), 
then we can take $A=\{p_{\ell}\}$ and the homology groups are trivial.
\end{proposition}

\begin{proof} Let $\Sigma=\sum_{a\in A}a$, in $\Z X$. We have $\Sigma *b=\Sigma$. We consider the chain homotopy
$h_n(x_1,...,x_n)=(\Sigma,x_1,...,x_n)$ (with the convention that $h_{-1}(1)=\Sigma$). 
This is a chain homotopy between $|A|Id$ and the zero map, i.e. we have 
$\partial_{n+1}h+ h\partial_{n} = |A|Id$.  Thus we
conclude that $|A|$ is an annihilator of homology ($|A|H_n(X)=0$). 
\end{proof}

\begin{remark}\label{Remark 5.4} 
\begin{enumerate} 
\item[(i)]
If we define $d_i: C_n\to C_{n-1}$ by:
$$d_0(x_1,...,x_n)= (x_2,...,x_n) $$
$$d_i(x_1,...,x_n)= (x_1,...,x_{i-1},x_i*x_{i+1},x_{i+2},...,x_n) \ \ for\ \ 0<i<n$$
$$and \ \ d_n(x_1,...,x_n)=(x_1,...x_{n-1})$$
then $(C_n,d_i)$ is a presimplicial module.
\item[(ii)] If $(X;*)$ is a monoid, we define degeneracy maps $s_0(x_1,...,x_n)= (1,x_1,...,x_n)$, and 
for $i>0$, $s_i(x_1,...,x_n)= (x_1,...,x_i,1,x_{i+1},...,x_n)$. Then $(C_n,d_i,s_i)$ is a simplicial module.
\end{enumerate}
\end{remark}

\subsection{Hochschild homology of a semigroup}
Let $(X;*)$ be a semigroup. We define a Hochschild chain complex $\{C_n,\partial_n\}$ as follows \cite{Hoch,Lod}:
$C_n(X)= \Z X^{n+1}$ and the Hochschild boundary \
$\partial_n : \Z X^n \to \Z X^{n-1}$ is defined by:
$$\partial(x_0,x_1,...x_n)= $$
$$\sum_{i=0}^{n-1}(-1)^i(x_0,...,x_{i-1},x_i*x_{i+1},x_{i+2},...,x_n) +$$
$$(-1)^n(x_n*x_{0},x_1,...x_{n-1})$$

The resulting homology is called the Hochschild homology of a semigroup $(X,*)$
and denoted by $H\!H_n(X)$ (introduced by Hochschild in 1945 \cite{Hoch}).
It is useful to define $C_{-1}=\Z$ and define $\partial_0(x)=1$ to obtain the augmented Hochschild 
chain complex and augmented Hochschild homology.

Again if $(X,*)$ is a monoid then dropping the last term gives an acyclic chain complex.

More generally (and similarly to group homology),
we check that if $(X,*)$ has a left unit $1_{\ell}$, then the chain homotopy 
$H_{\ell}(x_0,...,x_n)= (1_{\ell},x_0,...,x_n)$
satisfies $(\partial H_{\ell} + H_{\ell}\partial)(x_0,...,x_n)= (x_0,x_1,...,x_n)$, so the identity map
is chain homotopic to the zero map. 
For $(X,*)$ with a right unit $1_r$ we use the chain homotopy $H_r((x_0,...,x_n)= (-1)^{n+1}(x_0,...,x_n,1_r)$
to get a chain homotopy between the identity and the zero map.

Notice that dropping the last term in the definition of the boundary operation in 
Hochschild homology is like dropping the first and the last terms in
$\partial$ for the group homology of a semigroup (up to a grading shift).

\begin{remark} \begin{enumerate} 
\item[(i)]
If we define $d_i: C_n\to C_{n-1}$ by:
$$d_i(x_0,...,x_n)= (x_0,...,x_{i-1},x_i*x_{i+1},x_{i+2},...,x_n) \ \ for\ \ 0\leq i<n$$
$$and \ \ d_n(x_0,...,x_n)=(x_n*x_0,x_1,...x_{n-1}),$$
then $(C_n,d_i)$ is a presimplicial module.
\item[(ii)] If $(X,*)$ is a monoid, we define degeneracy maps 
for $0\leq i\leq n$ by the formula 
$s_i(x_0,...,x_n)= (x_0,...,x_i,1,x_{i+1},...,x_n)$. Then $(C_n,d_i,s_i)$ is a simplicial module.
\end{enumerate}
\end{remark}

\begin{remark}\label{Remark 5.6} 
To build a Hochschild chain complex we do not have to restrict ourselves to the case
of a semigroup $X$ or a semigroup ring $RX$. We can consider a general (associative) ring $A$
and our definitions still work due to the homogeneity of the boundary operation.
Thus we put $C_n(A)= A^{\otimes n+1}$, $d_i(a_0,...,a_n)= (a_0,..,a_i*a_{i+1},...a_n)$ for $0\leq i <n$, and
$d_n(a_0,...,a_n)= (a_n*a_0,a_1,...,a_{n-1})$. Notice that $d_id_{i+1}=d_id_i$ iff
$a_i*(a_{i+1}*a_{i+2})= (a_i*a_{i+1})*a_{i+2}$, that is, iff $*$ is associative.
\end{remark}

\section{Homology of distributive structures}\label{Section 6}

Recall that a shelf $(X,*)$ is a set $X$ with  a right self-distributive  binary operation
$*:X \times X \to X$ (i.e. $(a*b)*c= (a*c)*(b*c)$). 
\begin{definition}\label{Definition 6.1}
We define a (one-term) distributive chain complex ${\mathcal C}^{(*)}$ as 
follows:  $C_n=\Z X^{n+1}$ and the boundary operation $\partial^{(*)}_n: C_n \to C_{n-1}$ is given by:
$$\partial^{(*)}_n(x_0,...,x_n)= (x_1,...,x_n) +$$
 $$ \sum_{i=1}^{n}(-1)^i(x_0*x_i,...,x_{i-1}*x_i,x_{i+1},...,x_n).$$
The homology of this chain complex is called
a one-term distributive homology of $(X,*)$ (denoted by $H_n^{(*)}(X)$).
\end{definition}
We directly check that $\partial^{(*)}\partial^{(*)}=0$ (see Example 6.3 and Proposition \ref{Proposition 6.4}).

We can put $C_{-1}=\Z$ and $\partial_0(x)=1$. We have $\partial_0\partial_1^{(*)}=0$, so we 
obtain an augmented distributive chain complex and an 
augmented (one-term) distributive homology, $\tilde H^{(*)}_n$. As in the classical case we get:
\begin{proposition}\label{Proposition 6.4}
$
 H_n^{(*)}(X)=
 \begin{cases}
 \Z \oplus \tilde H^{(*)}_n(X) & n = 0 \\
 \tilde H^{(*)}_n(X) & \text{otherwise}
 \end{cases}
$
\end{proposition}

\begin{example}\label{Example 6.3} 
We check here that $\partial^{(*)}_1(\partial^{(*)}_2(x_0,x_1.x_2))=0$ is equivalent to 
$*$ being right self-distributive:
$$\partial^{(*)}_1(\partial^{(*)}_2(x_0,x_1,x_2))= \partial^{(*)}_1((x_1,x_2) - (x_0*x_1,x_2) + (x_0*x_2,x_1*x_2))=$$
 $$ x_2 - x_1*x_2 +$$
$$-x_2 + (x_0*x_1)*x_2 +$$
$$ x_1*x_2 - (x_0*x_2)*(x_1*x_2)=$$
$$(x_0*x_1)*x_2 - (x_0*x_2)*(x_1*x_2)\stackrel{distrib}{=}0$$
\end{example}

\begin{proposition}\label{Proposition 6.4}\ 
\begin{enumerate}
\item[(i)] Let $d_0(x_0,...x_n)= (x_1,...,x_n)$ and 
$d_i(x_0,...x_n)= (x_0*x_i,...,x_{i-1}*x_i,x_{i+1},...,x_n)$, for $0<i\leq n$. Then $(C_n,d_i)$ is a 
presimplicial module. In fact, $d_id_{i+1}=d_id_i$ for $i>0$ is equivalent to right self-distributivity.
\item[(ii)] Let $s_i(x_0,...x_n)= (x_0,...,x_{i-1},x_i,x_i,x_{i+1},...,x_n)$, then $(C_n,d_i,s_i)$ is a 
very weak simplicial module.
\item[(iii)] If $(X,*)$ is  a spindle, then $(C_n,d_i,s_i)$ is a weak simplicial module.
\end{enumerate}
\end{proposition}

\begin{proof} (i) This is a direct calculation and in the cases of $0=i\leq j$ and $i\leq j-1$ the equality 
$d_id_j=d_{j-1}d_i$ holds without any assumption on  $*$. The equality $d_id_{i+1}-d_id_i=0$ for 
$0<i=j-1$ is equivalent to right self-distributivity. We have:
$$(d_id_{i+1}-d_id_i)(x_0,...,x_n)= $$ 
$$ d_i((x_0*x_{i+1},...,x_{i-1}*x_{i+1},x_{i}*x_{i+1},x_{i+2},...,x_n) - 
(x_0*x_i,...,x_{i-1}*x_i,x_{i+1},x_{i+1},x_{i+2},...,x_n))= $$
$$((x_0*x_{i+1})*(x_{i}*x_{i+1}),...,(x_{i-1}*x_{i+1})*(x_{i}*x_{i+1}), x_{i+2},...,x_n) - $$
$$((x_0*x_i)*x_{i+1},...,(x_{i-1}*x_i)*x_{i+1},x_{i+2},...,x_n) \stackrel{distr}{=} 0.$$
(ii) A short calculation shows that conditions (2) and (3) of a very weak simplicial module hold 
without any assumption on $*$.\\
(iii) We check Condition (4') of Definition \ref{Definition 3.3}: $(d_is_i - d_{i+1}s_i)(x_0,...,x_n)=$
$$(d_i-d_{i+1})((x_0,...,x_{i-1},x_i,x_i,x_{i+1},...,x_n))= $$ 
$$ (x_0*x_i,...,x_{i-1}*x_i,{\bf x_i-x_i*x_i},x_{i+1},...,x_n) \stackrel{idemp}{=}0.$$
We notice that distributivity was not needed here, only the idempotency property of $*$.
\end{proof}
Proposition \ref{Proposition 6.4} is generalized in Lemma \ref{Lemma 7.1}.

\subsection{Computation of one-term distributive homology}\label{Subsection 6.1}
If $(X;*)$ is a rack, then the one-term (augmented) distributive chain complex is acyclic.
This may be the reason that this homology was not studied before. The first systematic calculations  
are given in \cite{P-S}. We observe there, in particular, that if there is given $b\in X$ in a shelf $(X;*)$ 
such that $*_b$ is invertible, then $\tilde H_n^{(*)}(X)=0$. To this effect, consider a chain homotopy 
$(-1)^{n+1}h_b$, where $h_b(x_0,...,x_n)=(x_0,...,x_n,b)$ to get 
$(\partial_{n+1}^{(*)}(-1)^{n+1}h_b + (-1)^{n}h_b\partial_{n}^{(*)})(x_0,...,x_n)= (x_0,...,x_n)*b$.
Thus the map $(x_0,...,x_n) \to (x_0,...,x_n)*b$ is chain homotopic to zero and if $*_b$ is invertible,
$\tilde H_n^{(*)}(X)=0$; compare Proposition \ref{Proposition 8.5}(v).

Below we show another result in this direction, motivated by an analogous observation from group homology 
(Proposition 5.3).
\begin{proposition}\label{Proposition 6.5}
 Assume that $(X;*)$ is a shelf which contains a finite right orbit  $A$,
that is, $A$ is a finite subset of $X$ such that for each $b\in X$, we have $*_b(A)= A*b=A$
(i.e. $*_b:A \to A$ is a bijection). Then $|A|$ annihilates $H_n(X)$. 
In particular, if $(X;*)$ has a left zero $p_{\ell }$ (i.e. $p_{\ell}*x=p_{\ell}$ for any $x\in X$),
then we can take $A=\{p_{\ell}\}$ and the (augmented) homology groups are trivial.

\end{proposition}

\begin{proof} The element $\sum_{a\in A}a \in \Z X$ is invariant under the right action, that is,
 $(\sum_{a\in A}a) *b=\sum_{a\in A}a$. We consider the chain homotopy
$$h(x_1,...,x_n)=((\sum_{a\in A}a),x_1,...,x_n) \text{ with the convention that } h(1)=\sum_{a\in A}a.$$
This is a chain homotopy between $|A|Id$ and the zero map, i.e. we have
$\partial_{n+1}h+ h\partial_{n} = |A|Id$.  Thus we
conclude that $|A|$ is an annihilator of homology ($|A|H_n(X)=0$).
\end{proof}
In Section 7, we introduce a multi-term distributive homology and Proposition \ref{Proposition 6.5} 
 can also be generalized to this,  
case, that is, for $\partial^{(a_1,...,a_k)}= \sum_{i=1}^k a_i\partial^{(*_i)}$ with $\sum_{i=1}^ka_i\neq 0$
and $A$ right invariant for any operation $*_i$.

In general, we conjecture in \cite{P-S} that one-term distributive homology is always torsion free. 
Thus in the case of 
Proposition \ref{Proposition 6.5} homology groups are conjectured to be trivial. In the special case of 
invertible $*$ (so $A=X$), we proved already at the beginning of this Subsection that the (augmented) 
homology groups are trivial (see also  \cite{P-S} and Corollary \ref{Corollary 8.2}(ii)).

\subsection{Computation for a shelf with $a*_gb=g(b)$}\

In \cite{P-S} we compute  the one-term distributive homology for a family of shelves  with a premiere example of 
a left trivial shelf  $(X;*_g)$, where $a*_gb=g(b)$ with $g^2=g$.
\begin{theorem}\cite{P-S}\label{Theorem 6.6}
$$\tilde H^{(*_g)}_n(X) \simeq \Z((g(X)-\{x_0\})\times X^n)$$ where $x_0$ is any fixed element of $g(X)$.
In other words, $\tilde H^{(*_g)}_n(X)$ is isomorphic to a free abelian group with basis $(g(X)-\{x_0\})\times X^n$.\\
For a finite $X$, we can write it as $\tilde H^{(*_g)}_n(X)= \Z^{(|g(X)|-1)|X|^n}$.

\end{theorem}

\begin{proof}
 We give a relatively short ``ideological" computation of $H_n^{(*_g)}(X)$ based on 
the short exact sequence of chain complexes introduced in Section 3 (compare
Remark \ref{Remark 3.4}). More precisely, let $F_0^{(t)}=F_0^{(t)}(C_n)= t_0(C_n)$, and 
$F_0^{(tD)}=F_0^{(tD)}(C_n)= span(t_0(C_n),s_0(C_{n-1}))$. We consider three nested 
chain complexes $F_0^{(t)} \subset F_0^{(tD)} \subset C_n$. The idea of our proof is to observe that 
$F_0^{(t)}$ has trivial boundary operations, $F_0^{(tD)}/F_0^{(t)}$ is acyclic, and $C_n/F_0^{(tD)}$ has 
trivial boundary operations. Finally, we have to study carefully the long exact sequence corresponding to 
the short exact sequence of chain complexes $0\to F_0^{(tD)} \to C_n \to C_n/F_0^{(tD)} \to 0$ to get the 
conclusion of the theorem.
In more detail, we are mostly interested in the case of $*_g$ from the theorem, 
but much of what follows applies in  more general setting.\
We have
$t_0(x_0,x_1,...,x_n)=(x_0,x_1,...,x_n) - (x_0*x_0,x_1,...,x_n)=(x_0-x_0*x_0,x_1,...,x_n)$; we 
use a ``bilinear notation".
We have $\partial t_0= 0$ as long as the equality $x*a = (x*x)*a$ holds\footnote{It holds if
a subshelf $A=*(X\times X)=\{z\in X\ | \ z=x*y \text{ for some  $x,y\in X$} \}$ is a spindle. 
For example, if there is a retraction $p: X \to A$ of $X$ to a spindle $A$ with $x_1*x_2=p(x_1)*p(x_2)$.
Two basic examples are: $a*_gb=g(b)$, $g^2=g=p$ and  $a*_fb=f(a)$, $f^2=f=p$; this idea is considered in \cite{P-S}.
} in $(X,*)$.
Thus we have:\\
(I) $H_n(F_0^t)= F_0^t=t_0C_n= \Z^{(|X|-|X/\sim|)|X|^{n}}$, 
 where $\sim$ is an equivalence relation on $X$ generated by $x \sim x*x$. For $*=*_g$, we can take as 
a basis of $H_n(F_0^t)= F_0^t$ elements $(x_0-g(x_0),x_1,...,x_n)$, and for a finite $X$, 
$H_n(F_0^t)=\Z^{(|X|-|g(X)|)|X|^{n}}$.\\
(II) For any shelf, $F_0^{(tD)}/F_0^{(t)}$ is acyclic. Namely, $s_0$ is a chain homotopy between
the identity and the zero map on $F_0^{(tD)}/F_0^{t}$. We have:
$$ \partial s_0 + s_0\partial = t_0 + s_0d_0 \equiv s_0d_0 \equiv Id\ on \ F_0^{(tD)}/F_0^{t}$$
(here we use the fact that in a shelf $d_0s_0=Id$, so $s_0d_0s_0=s_0$, thus $s_0d_0$ is the identity on
$F_0^{(tD)}/F_0^{t}$).\\
As a corollary, we have that the embedding $F^t_0 \to F_0^{(tD)}$ induces an isomorphism on homology.\\
(III) Consider now the chain complex $C_n/F_0^{(tD)}$. Here, for $a*_gb=g(b)$, $g^2=g$,
 the boundary operation is trivial as
$$\partial (x_0,x_1,...,x_n)= t_0(x_1,...,x_n)+\sum_{i=2}^n(-1)^i(g(x_i),...,g(x_i),x_{i+1},...,x_n)\in F_0^{(tD)}.$$
From this we conclude that $H_n(C_n/F_0^{(tD)})=C_n/F_0^{(tD)}$. As a basis of the group we can take 
elements $(x_0,x_1,...,x_n)$ with $x_0=g(x_0)$ and $x_1\neq x_0$. Thus the group is isomorphic to 
$\Z(g(X)\times (X-\{x_0\}|)\times X^{n-1})$ and for a finite $X$, the group is isomorphic to
 $\Z^{|g(X)|(|X|-1)|X|^{n-1}}$.\\
(IV) We consider the long exact sequence of homology corresponding to\\
 $0\to F_0^{(tD)}\to C_n \to C_n/F^{(tD)}\to 0$:
$$ ...\stackrel{b_*}{\rightarrow} H_n(F_0^{(tD)}) \to H_n(C) \to H_n(C/F^{(tD)}) \stackrel{b_*}{\rightarrow}
 H_{n-1}(F_0^{(tD)})\to... $$
We now show that the connecting homomorphism $b_*:H_n(C/F_0^{(tD)}) \to H_{n-1}(F_0^{(tD)})$ is an epimorphism. 
In fact, the element $(x_0,x_1-x_1*x_1,x_2,...,x_n)$ is a chain in $C_n$ but it is a cycle in $C_n/F_0^{(tD)}$. 
Thus its boundary $\partial(x_0,x_1-x_1*x_1,x_2,...,x_n)= (x_1-x_1*x_1,x_2,...,x_n)\in F_0^{(tD)}$, and 
this yields  
our connecting homomorphism $b: C_n/F_0^{(tD)} \to F_0^{(tD)}$, defined on the level of chains, with the 
image equal to $C_n^{(t)}$. 
However, because of (II), $b$ yields an epimorphism $b_*:H_n(C/F_0^{(tD)}) \to H_{n-1}(F_0^{(tD)})$.
Thus the long exact sequence of homology gives the short exact sequence:
$$0 \to H_n(C) \to H_n(C/F^{(tD)}) \stackrel{b_*}{\rightarrow} H_{n-1}(F_0^{(tD)})\to 0.$$
We now compute $H_n(C)$ as the kernel of $b_*$ to get the free abelian group with a basis 
obtained from the basis of $ H_n(C/F^{(tD)})$ by deleting elements of the form 
$(x_0,x_1-x_1*x_1,x_2,...,x_n)$ for fixed $x_0$. Thus, $H_n(X)$ is isomorphic to $ \Z((g(X))-\{x_0\})\times X^n)$
for $n>0$ and  $H_0(X)= \Z(g(X))$. If $X$ is finite we get $rank\ \tilde H_n = |g(X)|(|X|-1)|X|^{n-1}-(|X|-|g(X)|)|X|^{n-1}=
(|g(X)|-1)|X|^{n}$.
\end{proof}

We can make a small but useful generalization of Theorem \ref{Theorem 6.6} by considering
a new chain complex $C^{(d)}_n(X)$ obtained from $C_n^{(*_g)}(X)$ by taking, for any number $d$, the
boundary operation $\partial^{(d)} = d\partial^{(*_g)}$. See Theorem \ref{Theorem 9.1} and \cite{Pr-Pu} 
for further generalizations of the case $g=Id$.
\begin{corollary}\label{Corollary 6.7} Assume that $X$ is a finite set.
\begin{enumerate}
\item[(i)]If $d\neq 0$ then  $$\tilde H_n^{(d)}(X)= \Z^{(|g(X)|-1)|X|^n} \oplus \Z_d^{|X|^{n+1}-|g(X)|u_n},$$
where $u_n=u_n(|X|)=|X|^n - |X|^{n-1} +...+ (-1)^{n}|X|+(-1)^{n+1}$. In particular, if $g=Id$ we get 
$$\tilde H_n^{(d)}(X)= \Z^{(|X|-1)|X|^n} \oplus \Z_d^{u_n-(-1)^n}.$$
\item[(ii)] If $d=0$, then $H^{(d)}_n(X)= C^{(d)}_n(X)= \Z^{|X|^{n+1}}$.
\end{enumerate}
\end{corollary}
\begin{proof} (i) As long as $d\neq 0$ the free part of the homology does not depend on $d$, so we 
know the free part from Theorem \ref{Theorem 6.6}. We see that  the torsion part is 
$(\partial_{n+1}(C_{n+1}))\otimes \Z_d$, so for a finite $X$ we have to compute the rank of $\partial_{n+1}(C_{n+1})$.
We do this by observing that $$rk \partial_{n+1}(C_{n+1})+ rk \tilde H^{(*_g)}_n +rk \partial_{n}(C_{n})= 
rk C_n(X)= |X|^{n+1}.$$
For example, for $n=0$ we get $(|X|- |g(X)|) + (|g(X)|-1) +1 = |X|$  (we work with the 
reduced homology $\tilde H^{(*_g)}_n$).
Knowing initial data, the rank of homology, and the ranks of the chain groups we compute 
that $rk \partial_{n+1}(C_{n+1}) = |X|^{n+1}- |g(X)||X|^{n}+ |g(X)||X|^{n-1}+...+ (-1)^{n+1}|g(X)| =
|X|^{n+1}-|g(X)|u_n$, and the formula for homology is proven.\\
(ii) Boundary operations are trivial, so the formula follows.
\end{proof}

\section{Multi-term distributive homology}\label{Section 7} 

The first homology theory related to a self-distributive structure was constructed in early 1990s by 
Fenn, Rourke, and Sanderson \cite{FRS} and motivated by (higher dimensional) knot theory\footnote{The recent 
paper by Roger Fenn, \cite{Fenn} states:
"Unusually in the history of mathematics, the discovery of the homology and classifying
space of a rack can be precisely dated to 2 April 1990."}. 
For a rack $(X,*)$, they defined rack homology $H_n^R(X)$ by taking $C^R_n=\Z X^n$ and 
$\partial_n^R: C_n \to C_{n-1}$ is given by $\partial_n^R = \partial_{n-1}^{(*)}-\partial_{n-1}^{(*_0)}$.
Our notation has grading shifted by 1, that is, $C_n(X)= C^R_{n+1}= \Z X^{n+1}$. It is  routine to 
check that $\partial^R_{n-1}\partial_n^R=0$. However, it is an interesting question what properties 
of $*_0$ and $*$ are really used. With relation to the paper \cite{N-P-4} we noticed that 
it is distributivity again which makes $(C^R(X),\partial_n^R)$ a chain complex. More generally we 
observed that if $*_1$ and $*_2$ are right self-distributive and distributive  with respect to each other, 
then $\partial^{(a_1,a_2)}= a_1\partial^{(*_1)}+a_2\partial^{(*_2)}$ leads to a chain complex
(i.e. $\partial^{(a_1,a_2)}\partial^{(a_1,a_2)}=0$). 
Below I answer a more general question: for a finite set $\{*_1,...,*_k\}\subset Bin(X)$ and integers $a_1,...,a_k \in \Z$, 
when is $(C_n,\partial^{(a_1,...,a_k)})$ with $\partial^{(a_1,...,a_k)}=a_1\partial^{(*_1)}+...+a_k\partial^{(*_k)}$ 
a chain complex?
When is $(C_n,d_i^{(a_1,...,a_k)})$ a presimplicial set?
We answer these questions in Lemma \ref{Lemma 7.1}. In particular, for a distributive set $\{*_1,...,*_k\}$ 
the answer is affirmative.

\begin{lemma}\label{Lemma 7.1}\ 
\begin{enumerate}
\item[(i)] If $*_1$ and $*_2$ are right self-distributive operations, then  $(C_n,\partial^{(a_1,a_2)})$ is 
a chain complex if and only if the operations $*_1$ and $*_2$ satisfy:
$$(a*_1b)*_2c + (a*_2b)*_1c = (a*_2c)*_1(b*_2c) + (a*_1c)*_2(b*_1c).$$ 
We call this condition  {\it weak distributivity}.
\item[(ii)] We say that a set $\{*_1,...,*_k\}\subset Bin(X)$ is weakly distributive if each operation is 
right self-distributive and each pair of operations is weakly distributive (with two main cases:  
distributivity $(a*_1b)*_2c = (a*_2c)*_1(b*_2c)$ and chronological distributivity\footnote{I did not see this concept 
considered in literature, but it seems to be important in K.Putyra's work 
on odd Khovanov homology \cite{Put}; see also Proposition \ref{Proposition 7.2}.} 
$(a*_1b)*_2c=(a*_1c)*_2(b*_1c)$).
We have: $(C_n,d_i^{(a_1,...,a_k)})$ is a presimplicial set if and only if  
the  set $\{*_1,...,*_k\}\subset Bin(X)$ is weakly distributive.
\item[(iii)] $(C_n,\partial_n^{(a_1,...,a_k)})$ is a chain complex if and only if   
the  set $\{*_1,...,*_k\}\subset Bin(X)$ is weakly distributive.
\end{enumerate}
\end{lemma}
\begin{proof}
We have $\partial_{n-1}^{(a_1,a_2)}  \partial_n^{(a_1,a_2)}=$
$$(a_1\partial_{n-1}^{(*_1)}+a_2\partial_{n-1}^{(*_2)}) (a_1\partial_{n}^{(*_1)}+a_2\partial_{n}^{(*_2)}) =$$
$$a_1^2\partial_{n-1}^{(*_1)}\partial_{n}^{(*_1)} + a_2^2\partial_{n-1}^{(*_2)}\partial_{n}^{(*_2)}
+ a_1a_2(\partial_{n-1}^{(*_1)}\partial_{n}^{(*_2)} + \partial_{n-1}^{(*_2)}\partial_{n}^{(*_1)}) 
=$$
$$a_1a_2(\partial_{n-1}^{(*_1)}\partial_{n}^{(*_2)} + \partial_{n-1}^{(*_2)}\partial_{n}^{(*_1)})$$
To see that the condition $(a*_1b)*_2c + (a*_2b)*_1c = (a*_2c)*_1(b*_2c) + (a*_1c)*_2(b*_1c)$ is necessary, 
let us consider the case $n=2$. We have 
$$(\partial_{1}^{(*_1)}\partial_{2}^{(*_2)}) + \partial_{1}^{(*_2)}\partial_{2}^{(*_1)})(x_0,x_1,x_2)=$$
$$\partial_{1}^{(*_1)}((x_1,x_2) - (x_0*_2x_1,x_2)+ (x_0*_2x_2,x_1*_2x_2)) +$$ 
$$\partial_{1}^{(*_2)}((x_1,x_2) - (x_0*_1x_1,x_2)+ (x_0*_1x_2,x_1*_1x_2)) =$$
$$x_2 - x_1*_1x_2 - x_2 +(x_0*_2x_1)*_1x_2 + x_1*_2x_2 - (x_0*_2x_2)*_1(x_1*_2x_2) + $$
$$x_2 - x_1*_2x_2 - x_2 +(x_0*_1x_1)*_2x_2 + x_1*_1x_2 - (x_0*_1x_2)*_2(x_1*_1x_2) = $$
$$(x_0*_2x_1)*_1x_2 - (x_0*_2x_2)*_1(x_1*_2x_2) + (x_0*_1x_1)*_2x_2 - (x_0*_1x_2)*_2(x_1*_1x_2),$$
which is equal to zero iff  weak distributivity holds.\\
On the other hand, we show below that weak distributivity is sufficient to have  
$d_i^{(a_1,a_2)}d_{i+1}^{(a_1,a_2)}=d_i^{(a_1,a_2)}d_i^{(a_1,a_2)}$ for $0<i<n$ and is sufficient for
 $(C_n,d_i^{(a_1,a_2)})$ 
being a presimplicial module (the other needed equalities $d_id_j=d_{j-1}d_i$ for $i<j$ follow without using any 
special conditions). Namely, we have:
$$(d_i^{(a_1,a_2)}d_{i+1}^{(a_1,a_2)}-d_i^{(a_1,a_2)}d_i^{(a_1,a_2)})(x_0,...,x_n) =$$
$$d_i^{(a_1,a_2)}(a_1((x_0,...,x_i)*_1x_{i+1},x_{i+2},...,x_n) + a_2((x_0,...,x_i)*_2x_{i+1},x_{i+2},...,x_n) - $$
$$(a_1((x_0,...,x_{i-1})*_1x_i,x_{i+1},...,x_n) + a_2((x_0,...,x_{i-1})*_2x_i,x_{i+1},...,x_n))) =$$
$$a_1^2(((x_0,...,x_{i-1})*_1x_{i+1})*_1(x_i*_1x_{i+1}),x_{i+2},...,x_n) - $$
$$(((x_0,...,x_{i-1})*_1x_i)*_1x_{i+1},x_{i+2},...,x_n)) +$$
$$a_2^2(((x_0,...,x_{i-1})*_2x_{i+1})*_2(x_i*_2x_{i+1}),x_{i+2},...,x_n) -$$
$$(((x_0,...,x_{i-1})*_2x_i)*_2x_{i+1},x_{i+2},...,x_n)) +$$
$$a_1a_2(((x_0,...,x_{i-1})*_1x_{i+1})*_2(x_i*_1x_{i+1}),x_{i+2},...,x_n) +$$
$$a_1a_2(((x_0,...,x_{i-1})*_2x_{i+1})*_1(x_i*_2x_{i+1}),x_{i+2},...,x_n) - $$
$$a_1a_2(((x_0,...,x_{i-1})*_1x_i)*_2x_{i+1},x_{i+2},...,x_n))- $$
$$a_1a_2(((x_0,...,x_{i-1})*_2x_i)*_1x_{i+1},x_{i+2},...,x_n)) $$ 
which is equal to zero by the weak distributivity property. This completes our proof of (i); (ii) and (iii) 
follow from this directly.
\end{proof}

There is some justification for studying the concept of chronological-distributivity or weak distributivity,
as every semigroup $A$ (with the property: $xa=xb$, for every $x$, implies $a=b$) 
can be embedded as a chronological-distributive semigroup in $Bin(A)$ (compare Proposition \ref{Proposition 2.11}):
\begin{proposition}\label{Proposition 7.2}
\begin{enumerate}
\item[(i)] For $f:X \to X$ we define $*_f$ by $a*_fb=f(a)$. We have $*_f*_g= *_{gf}$ as 
$a*_f*_gb= (a*_fb)*_gb=g(f(a))= a*_{gf}b$. 
\item[(ii)]
For any pair of functions $f,g:X\to X$ the pair $(*_f,*_g)$ is chronologically distributive; namely we have:
$$(a*_fb)*_gc= g(f(a)),$$
$$(a*_fc)*_g((a*_fc)= g(a*_fc)=g(f(a)).$$
\item[(iii)] Any semigroup $A$ with the property: $xa=xb$, for every $x$, implies $a=b$,
 is a chronologically distributive subsemigroup of $Bin(A)$.
\item[(iv)] Any commutative semigroup $A$ with the property: $xa=xb$, for every $x$, implies $a=b$,
 is a distributive subsemigroup of $Bin(A)$.
\end{enumerate}
\end{proposition}
\begin{proof} The proof is a simple application of ideas from Proposition \ref{Proposition 2.11}. 
\end{proof}

\subsection{From distributivity to associativity}

We observed that to linearly combine two self-distributive operations into a new operation
we need weak distributivity. We can ask the similar question for associative operations, say $*_{\alpha}$ and
$*_{\beta}$ on $X$. For $\partial^{(a,b)}= a\partial^{\alpha}+ b\partial^{\beta},$ is it a boundary
operation? We consider group or Hochschild homology. A sufficient condition is that
$\partial^{\alpha}\partial^{\beta} = -\partial^{\beta}\partial^{\alpha}$.
This allows us not only to create linear combinations of boundary operations but also
to create a chain bicomplex using $\partial^{\alpha}$ horizontally and $\partial^{\beta}$ vertically.
The condition  $\partial^{\alpha}\partial^{\beta} = -\partial^{\beta}\partial^{\alpha}$ follows from:
$$(a*_{\alpha}b)*_{\beta}c + (a*_{\beta}b)*_{\alpha}c = $$
$$ a*_{\beta}(b*_{\alpha} c)+ a*_{\alpha}(b*_{\beta}c)$$

I do not know a good name for this so I will call it weak associativity 
(following the terminology from the distributive case), as it is a
combination of associativity $(a*_{\alpha}b)*_{\beta}c= a*_{\alpha}(b*_{\beta}c)$, and
chronological associativity (that is: $(a*_{\alpha}b)*_{\beta}c= a*_{\beta}(b*_{\alpha}c)$).

Of course weak associativity follows from each, associativity and chronological-associativity, separately.
\\ \ \ \\
\centerline{\psfig{figure=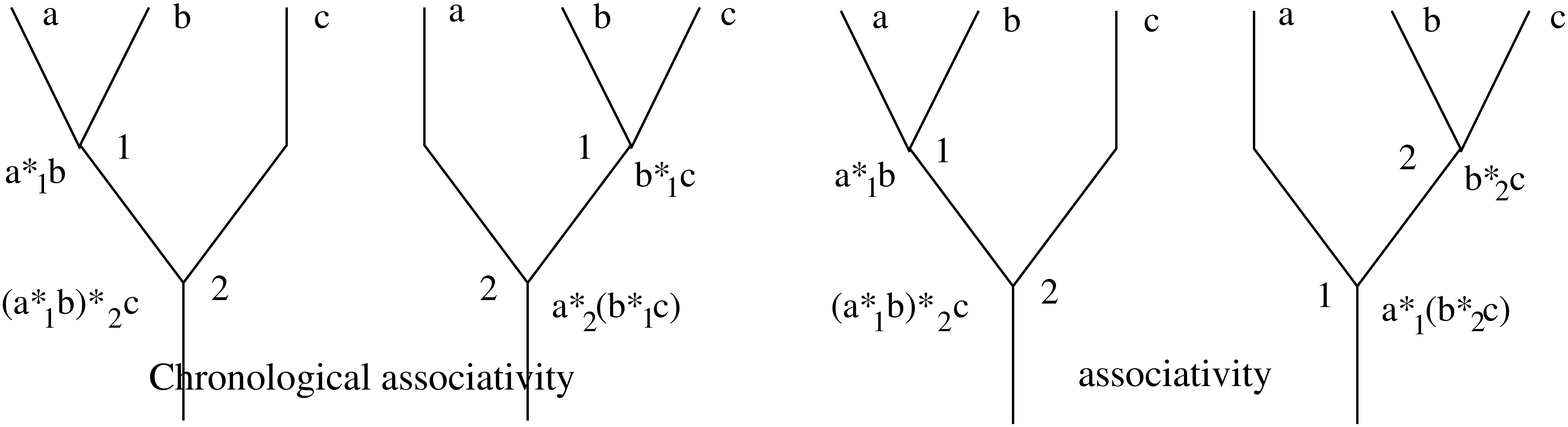,height=3.6cm}}
\  \\ 
{\bf Checking for $n=3$ and group homology}:\\
Let  $*_{\alpha}$ and $*_{\beta}$ be two associative operations on a set $X$. We have:
$$(\partial^{\beta}\partial^{\alpha}+ \partial^{\alpha}\partial^{\beta})(x_1,x_2,x_3)= $$
$$\partial^{\beta}((x_2,x_3)-(x_1*_{\alpha}x_2,x_3)+ (x_1,x_2*_{\alpha}x_3)- (x_1,x_2))+$$
$$\partial^{\alpha}((x_2,x_3)-(x_1*_{\beta}x_2,x_3)+ (x_1,x_2*_{\beta}x_3)- (x_1,x_2))=$$
$$x_3-x_2*_{\beta}x_3+x_2 -x_3 + (x_1*_{\alpha}x_2)*_{\beta}x_3 -x_1*_{\alpha}x_2 + $$
$$ x_2*_{\alpha}x_3 - x_1*_{\beta}(x_2*_{\alpha}x_3)+ x_1 -x_2 +x_1*_{\beta}x_2 -x_1 +$$
$$x_3-x_2*_{\alpha}x_3+x_2 -x_3 + (x_1*_{\beta}x_2)*_{\alpha}x_3 -x_1*_{\beta}x_2 + $$
$$ x_2*_{\beta}x_3 - x_1*_{\alpha}(x_2*_{\beta}x_3)+ x_1 -x_2 +x_1*_{\alpha}x_2 -x_1 =$$
$$(x_1*_{\alpha}x_2)*_{\beta}x_3 - x_1*_{\beta}(x_2*_{\alpha}x_3)+ $$
$$(x_1*_{\beta}x_2)*_{\alpha}x_3 - x_1*_{\alpha}(x_2*_{\beta}x_3)$$
which is equal to zero iff weak associativity holds.

\section{Techniques to study multi-term distributive homology}

\subsection{The remarkable map $f:X^{n+1}\to X^{n+1}$;\\ 
 $f(x_0,x_1,...,x_{n-1},x_n) = (x_0*x_1*...*x_n, x_1*...*x_n,...,x_{n-1}*x_n,x_n)$}\ \\

I noticed this very interesting map only in September 2010, but it looks similar to
the well known change of coordinates in homology of groups.

Let $*_0$ denote the trivial right action on $X$ (i.e. $a*_0b=a$), and let operations $*,*_1, *_2,...*_k$ be elements
of a  distributive submonoid of $Bin(X)$, that is, they are  right self-distributive operations on
a set $X$ which are distributive with respect to  another.\footnote{Historical note: 
The concept of a monoid of operations on a set $X$, $Bin(X)$, can be found in a classical literature, e.g. \cite{R-S}, 
however a multi-term distributive homology which followed, while motivated by rack and quandle homology,  
was only conceived in July 2010 at the end
of my visit to Gdansk and before Knots in Poland III. Seeds of the concepts
 were in the paper \cite{N-P-4} and the following:
\begin{observation}
\begin{enumerate}
\item[(i)] If $*: X\times X$ is a right self-distributive binary operation on $X$ then
$$*^k=
\underbrace{**...*}_{\mbox{$k$-times}}: X\times X \to X, $$
is also self-distributive.
\item[(ii)] If $*_1$ and $*_2$ are right self-distributive operations that are also right distributive with respect 
to each other then the composition $*_1*_2$ is right self-distributive.
\item[(iii)] $*_0$ defined by $a*_0b=a$ is right distributive with respect to any other operation
$(a*b)*_0c= (a*_0c)*(b*_0c)$ and $(a*_0b)*c= (a*c)*_0(b*c)$.
\item[(iv)] If two binary operations $*_1$ and $*_2$ are distributive  with respect to each other then
$$\partial^{(*_2)}\partial^{(*_1)}=-\partial^{(*_1)}\partial^{(*_2)}.$$
\item[(v)] If $*_1$ and $*_2$ are self-distributive and distributive   with respect to each other then
$\partial^{(a_1,a_2)}= a_1\partial^{(*_1)}+a_2\partial^{(*_2)}$ leads to a chain complex
(i.e. $\partial^{(a_1,a_2)}\partial^{(a_1,a_2)}=0$).
\end{enumerate}
 \end{observation}
}\\
Let $f=f^{(*)}: RX^n \to RX^{n} $ be given by\footnote{We use a standard convention for products 
in non-associative algebras, called the left normed convention, that is, whenever parentheses are
omitted in a product of elements $a_1$, $a_2,\ldots,$ $a_n$ of $X$ then
$a_1*a_2*\ldots *a_n=((\ldots ((a_1*a_2)*a_3)*\ldots)*a_{n-1})*a_n$ (left association),
for example, $a*b*c=(a*b)*c$).}: 
$$f(x_0,x_1,...,x_{n-1},x_n) = (x_0*x_1*...*x_n, x_1*...*x_n,...,x_{n-1}*x_n,x_n)$$
and $\partial^{(*)}(x_0,x_1,...,x_{n-1},x_n) = \sum_{i=0}^n (-1)^i(x_0*x_i,...,x_{i-1}*x_i, x_{i+1},...,x_n)$\
then    $ f^{(*)}\partial^{(*_2)}  = \partial^{(*_1)} f^{(*)} $ where $*_2= **_1$ 
(recall that the composition of operations is: $a (**_1) b = (a*b)*_1 b$), as
the following calculation demonstrates:\\
$f^{(*)}\partial^{(*_2)}(x_0,...,x_n)= $
$$ \sum_{i=0}^n (-1)^i(x_0*...*x_{i-1}*_2x_i*x_{i+1}*...*x_n,
x_{i-1}*_2x_i*x_{i+1}*...*x_n,x_{i+1}*...*x_n,...,x_n)$$
and
$\partial^{(*_1)}f^{(*)}(x_1,...,x_n)= $
$$\sum_{i=1}^n (-1)^i (x_0*...*x_{i-1}*x_i*_1x_i*x_{i+1}*...*x_n,
x_{i-1}*x_i*_1x_i*x_{i+1}*...*x_n,x_{i+1}*...*x_n,...,x_n).
$$
Here are interesting applications/special cases:\\
\begin{corollary}\label{Corollary 8.2}
\begin{enumerate}
\item[(i)]
Consider the multi-term boundary operation $\partial^{(a_1,...,a_n)}= \sum_{i=1}^k a_i\partial^{(*_i)}$,
then $f^{(*)}$ is a chain map from the chain complex on $\Z X^{n+1}$ with a composite boundary 
operation $*\circ\partial^{(a_1,...,a_n)}   \stackrel{def}{=} 
\sum_{i=1}^k a_i\partial^{(**_i)}$ to $(\Z X^{n+1},\partial^{(a_1,...,a_n)})$.\\
 Thus if $*$ is invertible (like in a rack) then
this chain map is invertible and induces an isomorphism of homology. In particular: 
\item[(ii)] If $\partial^{(*)}$ is a one-term operation with invertible $*$ then it has the same homology
as $\partial^{(*_0)}$ which is acyclic.
Here let us stress that we proved acyclicity for one-term homology for racks (for one-term homology
 we can prove acyclicity in a more general case: it suffices to
assume that there is $b$ such that $*_b$ is a bijection (as usually $*_b(a)=a*b$)).
(See Theorem \ref{Theorem 6.6} for examples of shelves that are not  racks and with a chain complex that is not acyclic). 
\item[(iii)] In classical (two-term) rack homology ($\partial = \partial^{(*_0)} -\partial^{(*)}$)
the above result gives an isomorphism with the chain complex $\bar\partial = \partial^{(\bar *)} -\partial^{(*_0)}$
which describes the classical homology of the dual complex ($\bar *$ in place of $*$)\footnote{The observation 
that rack or quandle
homology are the same for $(X;*)$ for $(X;\bar *)$ was proven first by S.Kamada and was known
to the authors of \cite{FRS}.}. 
\item[(iv)] More generally, we can consider any two-term complex with a boundary operation
 $a_1\partial^{(*_1)} + a_2\partial^{(*_2)}$ and for an invertible $*_1$
we get an isomorphic complex with $a_1\partial^{(*_0)} +a_2\partial^{(\bar*_1*_2)}$. 
This can be interpreted as saying that
any 2-term homology of racks is equivalent to the twisted homology \cite{CES-1} 
$\partial^T= t \partial^{0} -\partial^{1}$ (noninvertible $a_2$ gives slightly more possibilities).\\
Notice that, on the way from the twisted homology of $(X;*)$ and its dual $(X;\bar *)$, we also invert $t$.
\end{enumerate}
\end{corollary}

\subsection{Splitting multi-term distributive homology into degenerate and normalized parts}\label{Subsection 8.2}

For a quandle $(X;*)$ and its chain complex $(C_n,\partial^R)$, Carter, Kamada, and Saito at al. \cite{Car,CJKLS,CKS} 
considered the degenerate subcomplex and its quotient which they call {\it quandle chain complex}.
Litherland and Nelson \cite{L-N} proved that this complex splits. 
Their result extends to multi-spindle $(X;*_1,...,*_k)$ (that is, a multi-shelf with every operation idempotent).
 Our proof follows that given in \cite{N-P-2}.

Consider a multi-spindle $(X;*_1,...,*_k)$ and its chain complex $C_n(X)=\Z X^{n+1}$,
 $\partial^{(a_1,...,a_k)}= 
\sum_{i=1}^ka_i\partial^{(*_i)}$. Recall that we deal with a weak simplicial module $(C_n,d_i,s_i)$ with 
$d_i=d_i^{(a_1,...,a_k)}= \sum_{i=1}^ka_id^{(*_i)}_i$ and \\
 $s_i(x_0,...x_n)= (x_0,...,x_{i-1},x_i,x_i,x_{i+1},...,x_n)$. Thus, 
we know, in general, that $C_n^D=span(s_0C_{n-1},s_1C_{n-1},...,s_{n-1}C_{n-1})$ is a subchain complex of 
$(C_n,\partial^{(a_1,...,a_k)})$. This complex is usually not acyclic but it always splits. 
Let $C^{Norm}=C/C^D$ be the quotient complex, called the {\it normalized} complex of a multi-spindle.

\begin{theorem}\label{Theorem 8.3}
\begin{enumerate}
\item[(i)] Consider the short exact sequence of chain complexes:
$$0 \to C_n^D(X)  \to C_n(X) \to C_n^{Norm}(X) \to 0$$
Then this complex splits with a split map $\alpha : C^{Norm}(X) \to C_n(X)$ given by the formula:
$$\alpha(x_0,x_1,x_2,...,x_n)= (x_0,x_1-x_0,x_2-x_1,...,x_n-x_{n-1}).$$
We will the use multilinear convention as in \cite{N-P-2}, e.g. $\alpha(x_0,x_1,x_2)=$ 
$$(x_0,x_1-x_0,x_2-x_1)= (x_0,x_1,x_2)-(x_0,x_0,x_2) - (x_0,x_1,x_1)+ (x_0,x_0,x_1).$$
\item[(ii)] $H_n(X) = H_n^D(X) \oplus H_n^{Norm}(X)$.
\end{enumerate}
\end{theorem}

\begin{proof} (i) First observe that $\alpha$ is well defined since $\alpha(s_i(x_0,...,x_{n-1}))= 
(x_0,...,x_i-x_i,...,x_{n-1}) = 0$, so $\alpha (C_n^D)=0$. We also have  $\beta\alpha = Id_{C^{Norm}}$, because 
$(\alpha - Id)(C_n) \subset C^D_n$ and $\beta (C_n^D)=0$. This shows that $\alpha$ splits $\{C_n\}$ as a graded 
group. To show this split of a chain complex we should show that $\alpha$ is a chain map, that is,  
$\partial^{(a_1,...,a_k)}\alpha = \alpha\partial^{(a_1,...,a_k)}$. Of course it suffices to prove the relation 
$\partial^{(*_i)}\alpha = \alpha\partial^{(*_i)}$ for any $i$.
This follows from Lemma \ref{Lemma 8.4} below. Part (ii) follows directly from (i).
\end{proof}
 \begin{lemma}\label{Lemma 8.4}
\begin{enumerate}
\item[(i)] 
For any spindle $(X,*)$ and its related presimplicial module $(C_n,d_i)$ we have
$$d_i\alpha - \alpha d_i = r_{i-1}+ r_i, \ for \ 0\leq i \leq n, \ where$$ 
$$r_{−1} = r_0 = 0 \text{ and for } 0 < i < n:$$
 $$    r_i = −((x_0 , x_1 - x_0 , ..., x_i - x_{i-1} )* x_i , x_{i+2} - x_{i+1} , ..., x_n - x_{n-1})). $$
In particular, $r_0 = -(x_0 , x_2 - x_1 , ..., x_n - x_{n-1}))$ and \\
$r_{n-1} = -(x_0 , x_1 - x_0 , ..., x_{n-1} - x_{n-2} ) * x_{n-1}$.
\item[(ii)] $\partial^{(*)}_n \alpha - \alpha\partial^{(*)}_n = 0$.
\end{enumerate}
\end{lemma}
\begin{proof}
        We check immediately that $d_0 \alpha - \alpha d_0 = r_0$ and that $d_n \alpha - \alpha d_n = r_{n-1}.$
 Then, for $0 < i < n$ we compute:
$$(d_i\alpha - \alpha d_i)(x_0,x_1,...,x_n)= $$ 
$$d_i((x_0 , x_1 -x_0 ,...,x_n -x_{n-1} )- \alpha((x_0 , ..., x_{i-1})*x_i , x_{i+1} , ..., x_n ) =$$
$$((x_0 , x_1 -x_0 , ..., x_{i-1} -x_{i-2} )*(x_{i} -x_{i-1}), x_{i+1} -x_i , x_{i+2} -x_{i+1} , ..., x_n - x_{n-1})-$$
$$((x_0 , x_1 -x_0, ..., x_{i-1} -x_{i-2})*x_i , x_{i+1} - x_{i-1}*x_i , x_{i+2}-x_{i+1},...,x_{n}-x_{n−1}) = $$
$$-((x_0 , x_1 - x_0 , ..., x_{i-1} - x_{i-2} ) * x_{i-1} , x_{i+1} - x_i , x_{i+2} - x_{i+1} , ..., x_n - x_{n-1})+ $$
$$((x_0,x_1 -x_0,...,x_{i-1} -x_{i-2} )*x_i , x_{i+1} -x_i -x_{i+1} +x_{i-1} *x_i,x_{i+2}-x_{i+1},...,x_n -x_{n-1})=$$
$$r_{i-1} +((x_0 , x_1 -x_0 , ..., x_{i-1} -x_{i-2} )*x_i , (x_{i-1} -x_i )*x_i , x_{i+2} -x_{i+1},...,x_n -x_{n−1}) =$$
$r_{i-1} + r_i$, as needed.

(ii) follows from (i) as $\partial^{(*)}_n = \sum_{i=0}^n(-1)^id_i$ and $\sum_{i=0}^n (-1)^i(r_{i-1}+ r_i) = 0.$
\end{proof}

\subsection{Basic properties of multi-term distributive homology}\label{Subsection 8.3}

Let $(X;*_1,...,*_k)$ be a multi-shelf. We say that $A\subset X$ is a submulti-shelf if it is 
closed under all operations $*_i$. In particular, for an element $t\in X$, the set $\{t\}$ is a 
submulti-shelf iff it satisfies the idempotency condition for any operation ($t*_it=t$).
For a submulti-shelf $A$ we have the short exact sequence of chain complexes (recall that 
$\partial^{(a_1,...,a_k)}= \sum_{i=1}^ka_i\partial^{(*_i)}$ 
and to shorten notation we often write $\Sigma=\sum_{i=1}^ka_i$):
$$ 0 \to C_n(A) \to C_n(X) \to C_n(X,A) \to 0, \text{\ \  where } C_n(X,A)= C_n(X)/C_n(A).$$
\begin{proposition}\label{Proposition 8.5}
\begin{enumerate}
\item[(i)] Assume that for a submulti-shelf $A\subset X$ there is an operations-preserving retraction
$r: X \to A$. Then $r$ extends to a (chain complex) split of the above short exact sequence
$\tilde{r}:\Z X^{n+1} \to \Z A^{n+1}$. 
In particular, $H_n(X)=H_n(A) \oplus H_n(X,A)$.  
\item[(ii)] If $\{t\}\subset X$ is a one element submulti-shelf of $X$, then $X\to \{t\}$ is a 
multi-shelf retraction, thus, by (i)  
$C_n(X,\{t\})$ splits and $H_n(X)=H_n(\{t\}) \oplus H_n(X,\{t\})$. We think about  the  homology 
of $\{t\}$ as  a multi-shelf homology of a point, and call $H_n(X,\{t\})$ a {\it reduced homology}.
 \item[(iii)] Let $$\Sigma = \Sigma_{i=1}^k a_i\neq 0, \text{\ \ then }  H_n^{(a_1,...,a_n)}(\{t\})=
\begin{cases}
 \Z & n = 0 \\
 0 & n>0 \text{ even}\\
\Z_{\Sigma} &  \text{ $n$ is odd}
 \end{cases}
$$ 
and for $\Sigma = 0$, $H_n^{(a_1,...,a_n)}(\{t\})=\Z$ for any $n$.
\item[(iv)] Let $(X;*)$ be a shelf, $(x*t)*t=x*t$ for every $x\in X$, and $X*t$ 
be the orbit of the left action of $X$ on $t$ that is,  
$X*t=\{y\in X\ | \ y=x*t, \text{\ for some } x\in X\}$. Then $r_t=*_t: X \to X*t$ is a retraction; thus, by (i),
$H_n^{(*)}(X)=H_n^{(*)}(X*t) \oplus H_n^{(*)}(X,X*t)$.
\item[(v)]
Let $(X;*_1,...,*_k)$ be a multi-shelf.
Consider the map $h_t: C_n \to C_{n+1}$ given by $h_t(x_0,...,x_n)= (x_0,...,x_n,t)$, and the map 
$f_t=\sum_{i=1}^ka_i (*_i)_t$,
given by $f_t(x_0,...,x_n)= \sum_{i=1}^ka_i((x_0,...,x_n)*_it)$, 
then $(-1)^{n+1}h_t: C_n \to C_{n+1}$ is a chain homotopy between the map $f_t$
and the zero map.
\item[(vi)] Let  $(X;*_1,...,*_k)$ be a multi-shelf and $*_0$ the identity operation of $Bin(X)$.
Let $\partial^{(a_0,a_1,...,a_k)}=\sum_{i=0}^ka_i\partial^{(*_i)}$. Then $a_0Id_X$ is chain homotopic to 
$-f_t= -\sum_{i=1}^ka_i (*_it)$. 
\item[(vii)] Let $(X,*)$ be a shelf, and consider a rack boundary operation
 $\partial^R= \partial^{(*_0)}- \partial^{(*)}$,
then for any $t\in X$, we have $f_t=*_t$ is chain homotopic to $-Id_X$ and it is a (chain complex) retraction, 
thus $H^R_n(X,X*t)=0$ and $H^R_n(X)= H^R_n(X*t)$ for any $t\in X$ such that $(x*t)*t=x*t$ for every $x\in X$. 
A generalization of this observation
plays an important role in the computation of the 4-term homology of distributive lattices in \cite{Pr-Pu}.
\end{enumerate}
\end{proposition}
\begin{proof}
(i) If $i: A \to X$ is an embedding and $\tilde{i}: C_n(A) \to C_n(X)$ its linear extension to chain complexes, 
then $\tilde{r}\tilde{i}=Id_A$ and $\partial \tilde{r} = \tilde{r} \partial$, so $\tilde{r}$ is a 
 map that splits chain complex $C_n(X)$ and (i) of Proposition \ref{Proposition 8.5} follows.\\
(ii) It follows from (i) and idempotency $t*_it=t$.\\
(iii) $C_n(\{t\}) = \Z $ with basic element $(t,t,...,t)$. The chain complex reduces to:
$$ ... \stackrel{0}{\rightarrow} \Z  \stackrel{\times \Sigma}{\rightarrow} \Z 
\stackrel{0}{\rightarrow} \Z  \stackrel{\times \Sigma}{\rightarrow}
 \Z \stackrel{0}{\rightarrow} \Z  \stackrel{\times \Sigma}{\rightarrow} \Z \stackrel{0}{\rightarrow} \Z \to 0$$
and the homology follows immediately.\\
(iv) This follows from (i).\\
(v) We have $$\partial_{n+1}^{(a_1,...,a_k)}h_t - h_t\partial_{n}^{(a_1,...,a_k)} =
(-1)^{n+1}\sum_{i=1}^ka_i((x_0,...,x_0)*_it)$$
 and (v) follows.\\
(vi) This follows immediately from (v). \\
(vii) This is a consequence of (vi) but it should be stressed that it is a tautology for a rack (as then 
$X*t=X$ for any $t$). 
If $(X,*)$ is not a rack, that is, there is $t$ with $*_t$ not invertible, then we have a reduction in 
the computation of rack homology ($\partial^R= \partial^{(*_0)}- \partial^{(*)}$) \ from $X$ to $X*t$.
\end{proof}

We refer to \cite{Pr-Pu} for some useful generalizations of Proposition \ref{Proposition 8.5}.

We end this section by showing that 
the {\it reduced early degenerate complex} $(F_0,\{t\}) =s_0(C_{n-1})/C_n(\{t\})$ splits from  
the reduced chain complex $C(X,\{t\})$ of a multi-spindle $(X;*_1,...,*_n)$. The second factor 
$C(X,\{t\})/(F_0,\{t\})$ is called the {\it reduced early normalized chain complex} and denoted 
by $C^{eN}(X,\{t\})$. We also show how $F_0=\{F_n^0\}$ and $\{C_n\}$ are related.
\begin{proposition}\label{Proposition 8.6}
\begin{enumerate}
\item[(i)]
The short exact sequence of  multi-spindle chain complexes:
$$ 0 \to (F_n^0,\{t\}) \to C_n(X,\{t\}) \to C_n^{eN}(X,\{t\}) \to 0$$
splits with a split map $s_0p_0: C(X,\{t\}) \to (F_0,\{t\})$. where $p_0(x_0,x_1,...,x_n)=(x_1,...,x_n)$.
\item[(ii)] $s_0: C_{n-1}(X)\otimes \Z_{\sum_{i=1}^ka_k} \to s_0{C_{n-1}}\otimes \Z_{\sum_{i=1}^ka_k}$ 
yields an isomorphism on $mod (\sum_{i=1}^ka_k)$ homology.
\end{enumerate}
\end{proposition}
\begin{proof} Proposition \ref{Proposition 8.6} follows from  Lemma \ref{Lemma 8.7} (see also \cite{Pr-Pu} for 
further developments of these ideas).
\end{proof}
\begin{lemma}\label{Lemma 8.7}
\begin{enumerate}
\item[(i)] The map $s_0: C_n \to C_{n+1} $ is a chain homotopy between
$(\sum_{i=1}^k a_i) s_0p_0$ and the zero map. In particular $(\sum_{i=1}^k a_i)$ annihilates $H_n(F^0(X))$.
Furthermore, $s_0p_0$ is a chain that splits the chain complex of Proposition \ref{Proposition 8.6}(i).
\item[(ii)]
The map $p_0: C_n \to C_{n-1} $ is a chain homotopy between $(\sum_{i=1}^k a_i)p_0p_0$ and the zero map.
 Furthermore, $p_0p_0$ is a chain map.
\item[(iii)] If $(\sum_{i=1}^k a_i)=0$, then $(-1)^ns_0$ and $(-1)^np_0$ are chain maps (we write
$\sigma$ for $(-1)^ns_0$).
Furthermore, $p_0s_0= Id_{C_n}$ and  $s_0p_0=Id_{F^0}$. In particular, $\sigma: C_n \to F^0_{n+1}$ 
is an isomorphism of chain complexes.
\item[(iv)] More generally, $\sigma \otimes Id$  is a chain complex isomorphism
 $C_n(X)\otimes \Z_{\Sigma} \to F^0_{n+1}\otimes \Z_{\Sigma}$.
In particular, $H_n(X,\Z_{\Sigma})$ is isomorphic to $H_{n+1}(F_0,\Z_{\Sigma})$.
\end{enumerate}
\end{lemma}
\begin{proof}(i)  We use the fact that $d_0s_0=d_1s_0 = (\sum_{i=1}^k a_i)Id_{C_n}$ and that
$(C_n,d_i,s_i)$ is a weak simplicial module and, in particular, $d_is_0=s_0d_{i-1}$ for $i>1$.
Thus we have:
$$\partial^{(a_1,...,a_k)}s_0 + s_0\partial^{(a_1,...,a_k)}= \sum_{i=0}^{n+1}(-1)^{i}d_is_0 +
\sum_{i=0}^n(-1)^{i}s_0d_i= $$
$$(d_0s_0- d_1s0) + \sum_{i=2}^{n+1}(-1)^{i}d_is_0 + \sum_{i=0}^n(-1)^{i}s_0d_i= $$
$$ \sum_{i=2}^{n+1}(-1)^{i}s_0d_{i-1} + \sum_{i=0}^n(-1)^{i}s_0d_i= $$
 $$ \sum_{i=1}^{n+1}(-1)^{i+1}s_0d_{i} + \sum_{i=0}^n(-1)^{i}s_0d_i= s_0d_0 =
(\sum_{i=1}^k a_i) s_0p_0.$$
$(\sum_{i=1}^k a_i) s_0p_0$ is a chain map, and because $C_n$ is a complex of free groups, the map
$s_0p_0$ is a chain map.\\
We also can check directly that $s_0p_0: C_n \to C_n$ is a chain map that is a (chain) retraction 
to $F^0_n=s_0C_{n-1}$. First, $s_0p_0$ is the identity on $F^0_n$; further we have:
$$(d^{(*)}_0s_0p_0 -s_0p_0d^{(*)}_0)(x_0,x_1,x_2,...,x_n)= (x_0,x_1,x_2,...,x_n)- (x_1,x_1,x_2...,x_n),$$
$$(d^{(*)}_1s_0p_0 -s_0p_0d^{(*)}_1)(x_0,x_1,x_2,...,x_n)= (x_0,x_1,x_2,...,x_n)- (x_1,x_1,x_2...,x_n),$$
$$(d^{(*)}_is_0p_0 -s_0p_0d^{(*)}_i)(x_0,x_1,x_2,...,x_n)=0 \text{ for $i>1$  }.$$
Thus $\partial^{(*)}s_0p_0 = s_0p_0\partial^{(*)}$ and finally
$\partial^{(a_1,...,a_k)}s_0p_0 = s_0p_0\partial^{(a_1,...,a_k)}$.

(ii) We notice that $d_0p_0= (\sum_{i=1}^k a_i)p_0p_0$ and $d_ip_0=p_0d_{i+1}$. Thus:
$\partial^{(a_1,...,a_k)}p_0 + p_0\partial^{(a_1,...,a_k)}= (\sum_{i=1}^k a_i)p_0p_0$ and (ii) follows.\\
(iii) For $\sum_{i=1}^k a_i =0$ we directly see that $(-1)^ns_0$ and $(-1)^np_0$ are chain maps.
(iv) We see immediately that $s_0p_0=Id_{F_0}\mod \Sigma$ and $p_0s_0=Id_{C_n} \text{ mod }  \Sigma$.
\end{proof}

\section{Examples}\label{Examples 9}

 In this section  we illustrate our theory by various 
calculations of homology of multi-spindles. With the exception of racks (e.g. \cite{N-P-2,Nos,Cla}) 
no calculations were done before. 
We offer calculations of varying difficulties,
 starting from two-term homology. In Subsection 9.4 we make a detailed calculation using the following idea:
in the homology of a point, the chain groups $C_n(\{t\})$ are one-dimensional which makes the computation easy (see
Proposition \ref{Proposition 8.5} (iv)). For $|X|>1$ the chain groups grow exponentially, but there is one
case when the computation is not difficult, but still illuminating: the case of $|X|=2$ and
normalized homology, in which $C_n(X)$ is two-dimensional. For example, it works nicely for the group
homology of $\Z_2$ and for the Hochschild homology of $\Z(\Z_2)=\Z[x]/(x^2-1)$, or $\Z[x]/(x^2)$ (the underlying ring of
Khovanov homology). Here we show the calculation for a $4$-term distributive homology of a 4-spindle (in fact,
the maximal multi-spindle for $|X|=2$; see Subsection 9.3 and the 2-element Boolean algebra $B_1$). 

\subsection{The case of 2-term homology with  $\partial^{(a,d)}=a\partial^{(*_0)}+ d\partial^{(*_{\sim})}$}\

Define $*_{\sim}: X\times X \to X$ as the
 left trivial operation, that is $a*_{\sim}b=b$ (we will explain our notation 
in the section on Boolean algebras).

Below we consider the homology of the chain complex $(C(X);\partial^{(a,d)})$
where $\partial^{(a,d)}=a\partial^{(*_0)}+ d\partial^{(*_{\sim})}$. This generalizes Theorem \ref{Theorem 6.6} 
for $g=Id$ and is further generalized in \cite{Pr-Pu}.
\begin{theorem}\label{Theorem 9.1}\ 
\begin{enumerate}
\item[(1)] The chain complex $(C_n(X),\partial^{(a,d)})$ splits into three pieces:\\ 
(i) $C_n(\{t\})$, the chain complex of a point (we fix a point  $t\in X$), \\
(ii) $F_0(X,\{t\})=\{F^0_n(X,\{t\})\}=\{F^0_n/C_n(\{t\})\}=\{s_0C_{n-1}/C_n(\{t\})\}$, 
the reduced early degenerate chain complex, and\\
(iii)  $C_n^{eN}(X,\{t\})= C_n(X,\{t\})/F^0_n$, the reduced early normalized chain complex.\\
\item[(2)] If $a+d\neq 0$, then $H_n(\{t\})=
\begin{cases}
 \Z & \text{if $n = 0$} \\
0 & \text{if $n$ is even and } n>0 \\
 \Z_{a+d} & \text{if $n$ is odd} 
 \end{cases}
$\\
If $a+d=0$ then $H_n(\{t\})=\Z$.
\item[(3)] For a finite $X$, and $a$ or $d$ different from $0$ we have:\\  $H_n(F_0(X,\{t\}))=
\begin{cases}
\Z_{gcd(a,d)}^{u_n-1} & \text{ if $n$ is even} \\
\Z_{gcd(a,d)}^{u_n} & \text{ if $n$ is odd}
\end{cases} $\\
where $u_n(|X|)=u_n$ is 
defined by: 
$u_0=1$, $u_1=|X|-1$, and $u_n+u_{n-1} = |X|^n$, that is $u_n= |X|^n-u_{n-1}= |X|^n-|X|^{n-1}+...+(-1)^n=
\frac{|X|^{n+1}+(-1)^n}{|X|+1}$.
\item[(4)] For a finite $X$, and $a\neq 0$, we have  $$H_n(X,\{t\})/F_0)=
\Z_a^{u_{n+1}-u_n+(-1)^n}.$$
\item[(5)] If $a\neq 0$, and $a+d\neq 0$ then\\
$H_n(X)=
\begin{cases}
 \Z \oplus \Z_a^{|X|-1}& \text{if $n = 0$} \\
\Z_a^{u_{n+1}-u_n+1}\oplus \Z_{gcd(a,d)}^{u_{n}-1} & \text{if $n$ is even and } n>0 \\
 \Z_{a+d} \oplus \Z_a^{u_{n+1}-u_n -1}\oplus \Z_{gcd(a,d)}^{u_{n}} & \text{if $n$ is odd}
 \end{cases}
$ \\
The case of $a=0$ was already considered in Corollary \ref{Corollary 6.7}. The case of $a+d=0$ 
differs only from the general case in the factor $H_n(\{t\})$ so can be easily derived from (2)-(4).

\end{enumerate}
\end{theorem}
\begin{proof}
(1) This follows from Propositions \ref{Proposition 8.5} and \ref{Proposition 8.6}.\\
(2) This is a special case of Proposition \ref{Proposition 8.5}(iii).\\
(3) This follows from (4) and Lemma \ref{Lemma 8.7}(iii).\\
(4) First we notice that $\partial^{(a,d)}= a\partial^{(*_0)}$ in our chain group. The result follows from 
the fact that for $a=1$ we get an acyclic chain complex and from a careful analysis of the rank of 
$\partial_n(C^{eN})$.\\
(5) This is the summary of (2)-(4). 

\end{proof}
\subsection{Example: 3-term distributive homology of a spindle with $1_r$ and $0_r$}\

For any spindle $(X,*)$ we have the 3-element distributive set $\{*_0,*,*_{\sim}\}$; we check directly:
$$(x*y)*_{\sim}z=z \text{ and } (x*_{\sim}z)*(y*_{\sim}z)= z*z=z,$$
and
$$(x*_{\sim}y)*z= y*z \text{ and } (x*z)*_{\sim}(y*z)= y*z.$$
Thus we can consider 3-term distributive homology of a multi-spindle $(X;*_0,*,*_{\sim})$ with the 
boundary operation $\partial^{(a,c,d)}= a\partial^{(*_0)} + c\partial^{(*)} + d \partial^{(*_{\sim})}$.
Computation of this homology, in general, is a difficult problem as it contains quandle homology as 
a special case. However, for $*$ with a right unit $1_r$ (i.e. $x*1_r=x$), and a right projector $0_r$ 
(i.e. $x*0_r=0_r$), the solution can be obtained in a manner similar to that 
of Theorem \ref{Theorem 9.1}. Namely, we have:
\begin{theorem}\label{Theorem 9.2}\
\begin{enumerate}
\item[(1)] The chain complex $(C_n(X),\partial^{(a,c,d)})$ splits into three pieces:\\
(i) $C_n(\{1_r\})$, the chain complex of a point (we fix a point  $1_r$), \\
(ii) $F_0(X,\{1_r\})=\{F^0_n/C_n(\{1_r\})\}=\{s_0C_{n-1}/C_n(\{1_r\})\}$, 
the reduced chain complex of  early degenerate elements, and\\
(iii)  $C_n^{eN}(X,\{t\}) =C_n(X,\{1_r\})/F^0_n$, the reduced early normalized chain complex.\\
\item[(2)] If $a+c+d\neq 0$, then $H_n(\{t\})=
\begin{cases}
 \Z & \text{if $n = 0$} \\
0 & \text{if $n$ is even and } n>0 \\
 \Z_{a+c+d} & \text{if $n$ is odd}
 \end{cases}
$\\
If $a+c+d=0$ then $H_n(\{1_r\})=\Z$.
\item[(3)] For a finite $X$, and $a$, $c$, or $d$  different from $0$, we have:\\  $H_n(s_0C_{n-1}/C_n(\{1_r\})=
\begin{cases}
\Z_{gcd(a,c,d)}^{u_n-1} & \text{ if $n$ is even} \\
\Z_{gcd(a,c,d)}^{u_n} & \text{ if $n$ is odd}
\end{cases} $\\
where $u_n=u_n(|X|)= |X|^n-|X|^{n-1}+...+(-1)^n$, as in Theorem \ref{Theorem 9.1}.
\item[(4)] For a finite $X$, and $a$ or $c$ different from $0$, we have  $$H_n((X,\{1_r\})/F_0)=
\Z_{gcd(a,c)}^{u_{n+1}-u_n+(-1)^n}.$$
\item[(5)] If $a$ or $c$ $\neq 0$, and $a+c+ d\neq 0$ then\\
$H_n(X)=
\begin{cases}
 \Z \oplus \Z_{gcd(a,c)}^{|X|-1}& \text{if $n = 0$} \\
\Z_{gcd(a,c)}^{u_{n+1}-u_n+1}\oplus \Z_{gcd(a,c,d)}^{u_{n}-1} & \text{if $n$ is even and } n>0 \\
 \Z_{a+c+d} \oplus \Z_{gcd(a,c)}^{u_{n+1}-u_n -1}\oplus \Z_{gcd(a,c,d)}^{u_{n}} & \text{if $n$ is odd}
 \end{cases}
$ \\
The case $a=c=0$ was already described in Corollary \ref{Corollary 6.7}.
The case $a+c+d=0$  differs from other cases only at $H_n(\{1_r\})$. 
\end{enumerate}
\end{theorem}
\begin{proof} The proof is a refinement of the proof of Theorem \ref{Theorem 9.1}. We first consider 
the chain homotopy $h_{1_r}$ to get:
$$(-1)^{n+1}(\partial^{(a,c,d)}h_{1_r} - h_{1_r}\partial^{(a,c,d)})(x_0,...,x_n)= a(x_0,...,x_n) + 
(c+d)(1_r,...,.1_r),$$ 
and the chain homotopy $h_{0_r}$ to  get:
$$(-1)^{n+1}(\partial^{(a,c,d)}h_{0_r} - h_{0_r}\partial^{(a,c,d)})(x_0,...,x_n)= (a+c)(x_0,...,x_n) + 
d(0_r,...,.0_r),$$
From this we conclude that $H_n(C,\{1_r\})$ is annihilated by $gcd(a,c)$. Further we proceed like in the proof 
of Theorem \ref{Theorem 9.1}; see \cite{Pr-Pu} for details.
\end{proof} 
\subsection{Example: 4-term normalized distributive homology of the 2-element Boolean algebra}\label{Subsection 9.3} 

Our first interesting example of a distributive monoid is given by a distributive lattice (e.g. Boolean algebra)
$({\mathcal L},\cup,\cap)$ because lattice operations $\cup$ and $\cap$ form a distributive set. 
(We refer to \cite{B-D,Gra,Si,Tra} for an extensive coverage of distributive lattices\footnote{In our language 
a distributive lattice is a multi-spindle $({\mathcal L},\cup,\cap)$ with commutative and associative operations,
satisfying absorption axioms: $(a\cup b)\cap b = b = (a\cap b)\cup b$.} and Boolean algebras).
In this paper, we 
denote these binary operations by $*_{\cup}$ and $*_{\cap}$. The distributive monoid spanned by these operations
is a commutative monoid of 4 idempotent elements: $*_0$ - identity element, $*_{\cup}$, $*_{\cap}$, 
and the composition $*_{\sim}= *_{\cap} *_{\cup}$. 
One can present the monoid as: 
$$\{*_{\cup},*_{\cap}\ | \ *_{\cup}*_{\cap}=*_{\cap}*_{\cup}, *_{\cup}*_{\cup}=*_{\cup}, 
*_{\cap}*_{\cap}=*_{\cap}\}.$$
Notice that $*_{\sim}$ is the left trivial operation, $a*_{\sim}b=b$.

Using our 4-element distributive monoid we can consider the 4-term boundary operation:
$\partial^{(a,b,c,d)}: C_n({\mathcal L}) \to C_{n-1}({\mathcal L})$, where $C_n({\mathcal L})= \Z{\mathcal L}^{n+1}$ 
and $\partial=\partial^{(a,b,c,d)} =
a\partial^{(*_0)}+ b\partial^{(*_{\cup})} + c\partial^{(*_{\cap})} + d\partial^{(*_{\sim})}$.

The computation of the four-term distributive homology of ${\mathcal L}$ is 
generally difficult, but it is done fully in \cite{Pr-Pu}; see Theorem \ref{Theorem 9.5}. 
For normalized homology 
we are able to make a very elementary (and illuminating, in my opinion) calculation in 
the simplest nontrivial case of $B_1=\{0,1\}$, the two element Boolean algebra 
of subsets of the one element set. This case is approachable because $C^{Norm}_n(B_1)$ is 2-dimensional for any $n$.
Choose the basis $e_n= (0,1,0,1,0,...)$, $e'_n=(1,0,1,0,1,...)$ of $C^{Norm}_n(B_1)=C_n(B_1)/C_n^D(B_1)$.
To be able to deduce homology, it is enough to write $\partial$ in this basis.
We have to consider the case of $n$ even and odd separately.\\

\subsection{Detailed calculation of the quandle homology of $\bf{X=B_1}$}

$$\partial^{(*_0)}(e_n) = (-1)^{n}e_{n-1}+ e'_{n-1}= (-1)^n\partial^{(*_0)}(e_n').$$

For $n$ even we have $e_n= (0,1,...,1,0)$ and  $e_n'= (1,0,...,0,1)$; then
$$\partial^{(*_{\cup})}(0,1,...,1,0)= (-1)^{n}e_{n-1}$$
$$\partial^{(*_{\cup})}(1,0,...,0,1)= e_{n-1}$$

For $n$ odd we have $e_n= (0,1,...,0,1)$ and $e'_n= (1,0,...,1,0)$; then
$$\partial^{(*_{\cap})}(0,1,...,1,0,1)= (-1)^{n} e_{n-1} +e_{n-1}'$$
$$\partial^{(*_{\cap})}(1,0,...,0,1,0)= 0.$$
$$\partial^{(*_{\sim})}(e_n)= \partial^{(*_{\sim})}(e'_n)=0.$$
For $\partial^{(a,b,c,d)}= a\partial^{*_0} + b\partial^{*_{\cup}} + c\partial^{*_{\cap}} +d\partial^{*_{\sim}},$ 
and for $n$ even
$$ \partial^{(a,b,c,d)}(e_n)= (-1)^n(a+b)e_{n-1} + (a+c)e'_{n-1} =(-1)^n\partial^{(a,b,c,d)}(e'_n).$$
For $n$ odd:
$$\partial^{(a,b,c,d)}(e_n)= (-1)^n(a+c)e_{n-1} +(a+c)e'_{n-1} ,$$
$$\partial^{(a,b,c,d)}(e'_n)= (a+b)(e_{n-1}+ (-1)^ne'_{n-1}) .$$
Therefore, we have the following matrices of relations in $C^{Norm}_n/\partial(C^{Norm}_{n+1})$. For $n$ even:
\[
\left(
\begin{array}{cc}
(-1)^n(a+b) & a+c  \\
a+b & (-1)^n(a+c)  \end{array}
\right)\]

For $n$ odd:
\[
\begin{pmatrix}
(-1)^n(a+c) & a+c  \\
a+b &  (-1)^n(a+b)
\end{pmatrix} \]

From this we get:
\begin{proposition}\label{Proposition 9.3}\ 
\begin{enumerate}
\item[(i)] $C^{Norm}_n/\partial(C^{Norm}_{n+1}) = \Z \oplus \Z_{gcd(a+b,a+c)}$.
\item[(ii)] For $n>0$, $\partial (C^{Norm}_n) = \Z$, unless $a+b=a+c=0$ in which case
$\partial (C^{Norm}_n) =0$.
\item[(ii)] For $n>0$, $H^{Norm}_n(B_1)= \Z_{gcd(a+b,a+c)}$, unless $a+b=a+c=0$ in which case
$H^{Norm}_n(B_1)= \Z\oplus \Z$. \\ 
$H^{Norm}_0(B_1)= C^{Norm}_0/\partial(C^{Norm}_{1}) = \Z \oplus \Z_{gcd(a+b,a+c)}$.
\end{enumerate}
\end{proposition}
In the proof we use the standard but important observations that \\
(i) $rank(H_n) + rank (\mathrm{Im} \partial_{n+1})+ rank (\mathrm{Im} \partial_{n}) = X^n$, and \\
(ii) $tor H_n(X) = tor (\Z X^{n+1}/\mathrm{Im} (\partial_{n+1})).$ \\


The degenerate part of the homology is much more difficult. We started with computer experiments 
(with the help of Michal Jablonowski and Krzysztof Putyra) and eventually proved the following: 
\begin{theorem}\cite{Pr-Pu}\label{Theorem 9.4}
 Assume that $a+b+c+d \neq 0$ and $a+b\neq 0$ or $a+c\neq 0$ then $rank H_n^D(B_1)=0$ and
$$ H_n^D(B_1)=
 \left\{ \begin{array}{rl}
  \Z_{gcd(a+b,a+c)}^{a_n-1} \oplus \Z_{gcd(a+b,a+c,c+d)}^{a_n-1} &\mbox{ if $n$ is even} \\
 \Z_{a+b+c+d} \oplus \Z_{gcd(a+b,a+c)}^{a_n-1} \oplus \Z_{gcd(a+b,a+c,c+d)}^{a_n} & \mbox{ if $n$ is odd}
       \end{array} \right.
$$
In the formula above $a_n=u_n(2) $ (see Theorem \ref{Theorem 9.1}), that is, 
$a_0=a_1=1$, $a_n+a_{n-1}= 2^n$, and thus
$a_n= 2a_{n-1}+(-1)^n=2^n-2^{n-1}+...+ (-1)^n = \frac{2^{n+1}+(-1)^n}{3}$.
\end{theorem}

\subsection{More about homology for
 $\bf{ \partial^{(a,b,c,d)}= a\partial^{*_0} + b\partial^{*_{\cup}} + c\partial^{*_{\cap}} +d\partial^{*_{\sim}}}$}
\ \\
Two months after a June seminar talk I gave at Warsaw Technical University, we found a general formula 
for the four-term distributive homology of any finite distributive lattice. 
For $b=c=0$ it gives Theorem \ref{Theorem 9.1}. To formulate Theorem \ref{Theorem 9.5} 
we need some basic terminology:
let ${\mathcal L}$ be a distributive lattice; we say that
  an element a of ${\mathcal L}$ is join-irreducible if for any
decomposition $a = b \cup c$, we have $a = b$ or $a = c$. Let $J({\mathcal L})$ be the set of non-minimal (different from
$\emptyset$),
join-irreducible elements in ${\mathcal L}$ and $J$ its cardinality. In what follows $L$ denotes the cardinality of
${\mathcal L}$.
If ${\mathcal L}$ is finite, then $J$ is equal to the length of every maximal chain
in ${\mathcal L}$ (see Corollary 14 in \cite{Gra}). 
\begin{theorem}\cite{Pr-Pu}\label{Theorem 9.5}.  Let ${\mathcal L}$ be a finite distributive lattice.
Assume for simplicity that $a+b+c+d\neq 0$, $a+b$ or $a+c$ is not equal to $0$, and one of $a$, 
$b$ and $c$ is not equal to $0$. Then $ H^{(a,b,c,d)}_n({\mathcal L})=$
$$ \Z \oplus \Z_{gcd(a+b,a+c)}^{J}\oplus \Z_{gcd(a,b,c)}^{L-J-1}  \text{ if $n=0$ }, $$
$$  \Z_{gcd(a+b,a+c)}^{Ju_n(2)} \oplus \Z_{gcd(a,b,c)}^{u_{n+1}(L)-u_n(L)+1-Ju_n(2)} 
\oplus \Z_{gcd(a+b,a+c,c+d)}^{Ju_n(2)-J}
 \oplus \Z_{gcd(a,b,c,d)}^{u_n(L)-1-Ju_n(2)+J}$$
 if $n$ is even, and 
$$ \Z_{a+b+c+d} \oplus  
\Z_{gcd(a+b,a+c)}^{Ju_n(2)} \oplus \Z_{gcd(a,b,c)}^{u_{n+1}(L)-u_n(L)-1-Ju_n(2)} \oplus \Z_{gcd(a+b,a+c,c+d)}^{Ju_n(2)}
\oplus \Z_{gcd(a,b,c,d)}^{u_n(L)-Ju_n(2)} $$
 if $n$ is odd.
\end{theorem}

\subsection{Generalized lattices}

Our computation in \cite{Pr-Pu} of the four-term homology of a distributive lattice can be partially
generalized and this justifies an introduction of the following multi-spindle, 
in which commutativity or associativity of operations are not assumed.
\begin{definition}
A {\it generalized lattice} $(X;*_1,*_2)$ is a set with two binary operations which satisfy the
following three conditions:\\
(1) Each operation is right self-distributive.\\
(2) Absorption conditions hold: $(a*_1b)*_2b=b=(a*_2b)*_1b$
 (in particular each action satisfies the idempotency condition).\\
(3) $(a*_1b)*_1b= a*_1b$  and $(a*_2b)*_2b= a*_2b$ \\
If additionally our operations are right distributive with respect to each other:\\
(4) $(a*_1b)*_2c= (a*_2c)*_1(b*_2c)$ and  $(a*_2b)*_1c= (a*_1c)*_2(b*_1c)$,
we call $(X;*_1,*_2)$ a {\it generalized distributive lattice}.
\end{definition}
We should comment here that absorption implies that $*_1*_2=*_2*_1=*_{\sim}$ and idempotency
of each operation $a*_1a=a=a*_2a$ (we have: $((a *_1 a) *_2 a) *_1 a = a*_1 a$  (absorption for $b=a$, i.e.
 $(a *_1 a) *_2 a = a $). We also have $((a *_1 a) *_2 a) *_1 a  =  a$ (absorption for $b=a *_1 a)$;
thus $a *_1 a = a$.
The monoid in $Bin(X)$ generated by $(*_1,*_2)$ is isomorphic to the four element monoid from classical
(distributive) lattices (Subsection 9.3).

\section{Motivation from Knot Theory}

The fundamental result in combinatorial knot theory, envisioned by Maxwell and proved by 
Reidemeister and Alexander and Briggs around 1927, is that links in $R^3$ are equivalent (isotopic) 
if and only if their diagrams are related by a finite number of local moves (now called Reidemeister moves). 
Three Reidemeister moves are illustrated in Figures 10.2- 10.4; 
see \cite{Prz-1,Prz-2} for an early history of knot theory.
Thus, one can think about classical knot theory as analyzing knot diagrams modulo Reidemeister moves. 
One can, naively but successfully, construct knot invariants as follows: choose a set $X$ with a binary 
operation $*: X \times X \to X$, and consider ``colorings" of arcs of an oriented diagram $D$ (arcs 
are from undercrossing to undercrossing) by elements of $X$ so that,  
at every crossing, the coloring satisfies the condition from Fig. 10.1. This gives a different condition for a 
positive and negative crossing, which can be put together as in Fig. 10.1 (iii) (here only the 
overcrossing has to be oriented and, of course, we need an orientation of the plane of the projection). 
 We interpret the use of the operation $*$ as saying that an overcrossing is acting on an 
undercrossing. We define a diagram invariant $col_X(D)$ as a cardinality of a set 
of all allowed colorings of $D$, that is, 
$col_X(D)= |\{f:arcs(D) \to X\ | \ f \text{ satisfies the rules of Fig. 10.1} \}|$.
\\ \ \\
\centerline{\psfig{figure=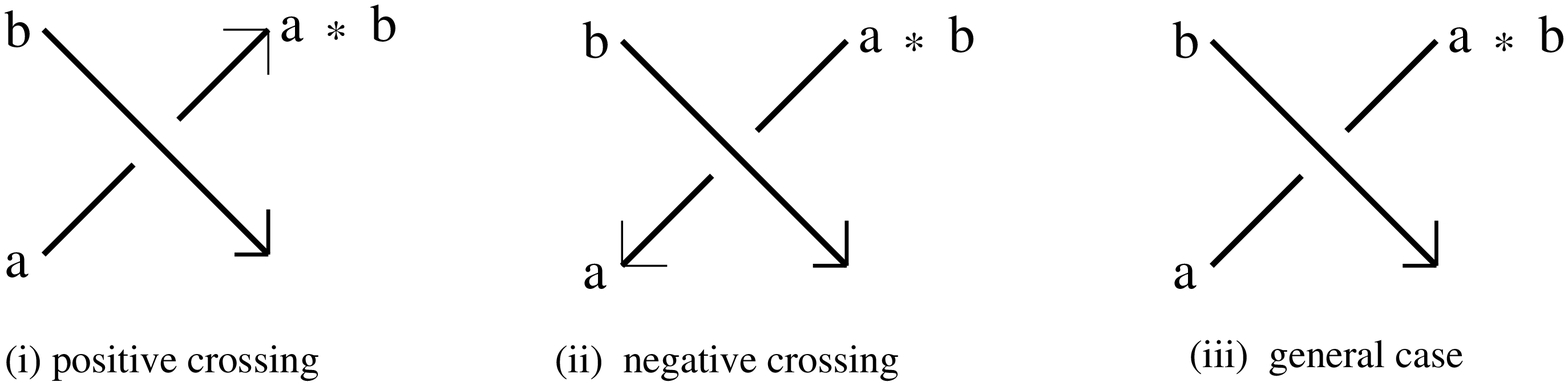,height=3.3cm}}
\centerline{Figure 10.1; local coloring by $(X,*)$}

In order to be a link invariant, $col_X(D)$ should be invariant under the Reidemeister moves, which provides 
motivation for the  axioms of a quandle.
\begin{enumerate}
\item[$(R_1)$] The first Reidemeister move of Figure 10.2 requires the idempotency $a*a=a$ (left part), and 
in the case of the right part we need a unique solution $x=a$ for the equation $x*a=a$, this follows 
from the idempotency and invertibility of $*$. 
\end{enumerate}
\centerline{\psfig{figure=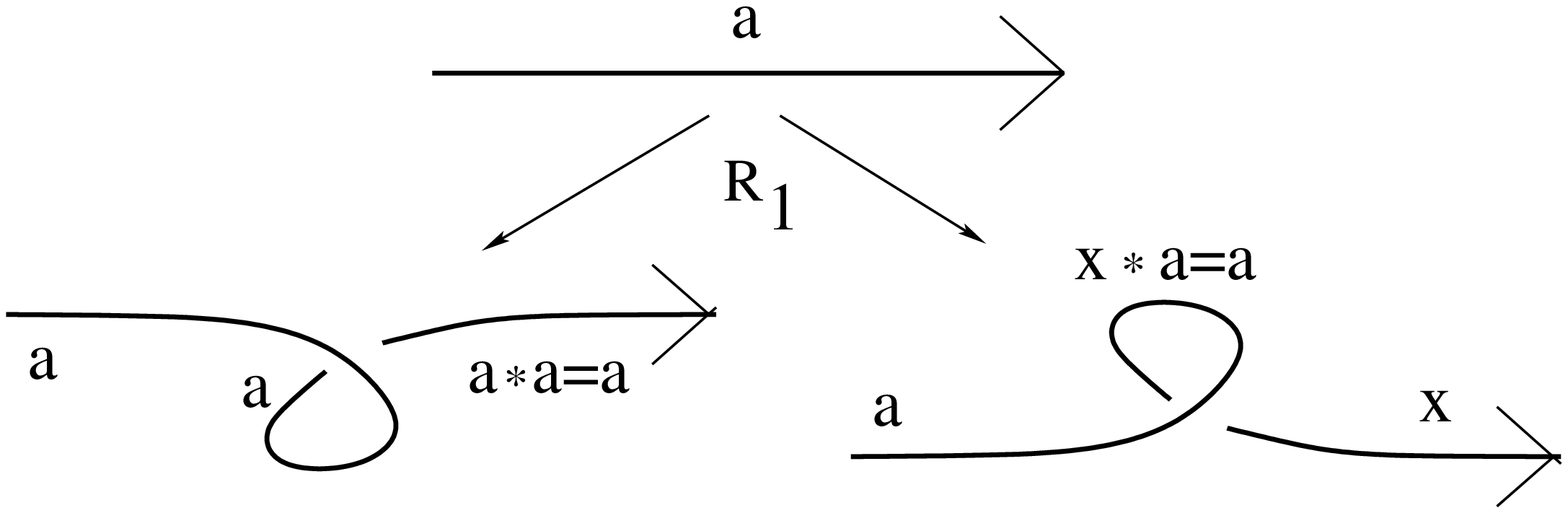,height=3.8cm}}\ \\ \ \\
\centerline{Figure 10.2; $col_X(R_1(D))= col_X(D)$}
\begin{enumerate}
\item[$(R_2)$] The second Reidemeister move requires invertibility of $*$. In fact, the move from 
Fig.~10.3(i) requires $*_b$ to be injective ($a*b=a'*b \Rightarrow a=a'$) and that of Fig.~10.3(ii) requires $*_b$ 
to be bijective (for any $a,b$ there is the unique $x$ such that $a=x*b$).
\end{enumerate}
\centerline{\psfig{figure=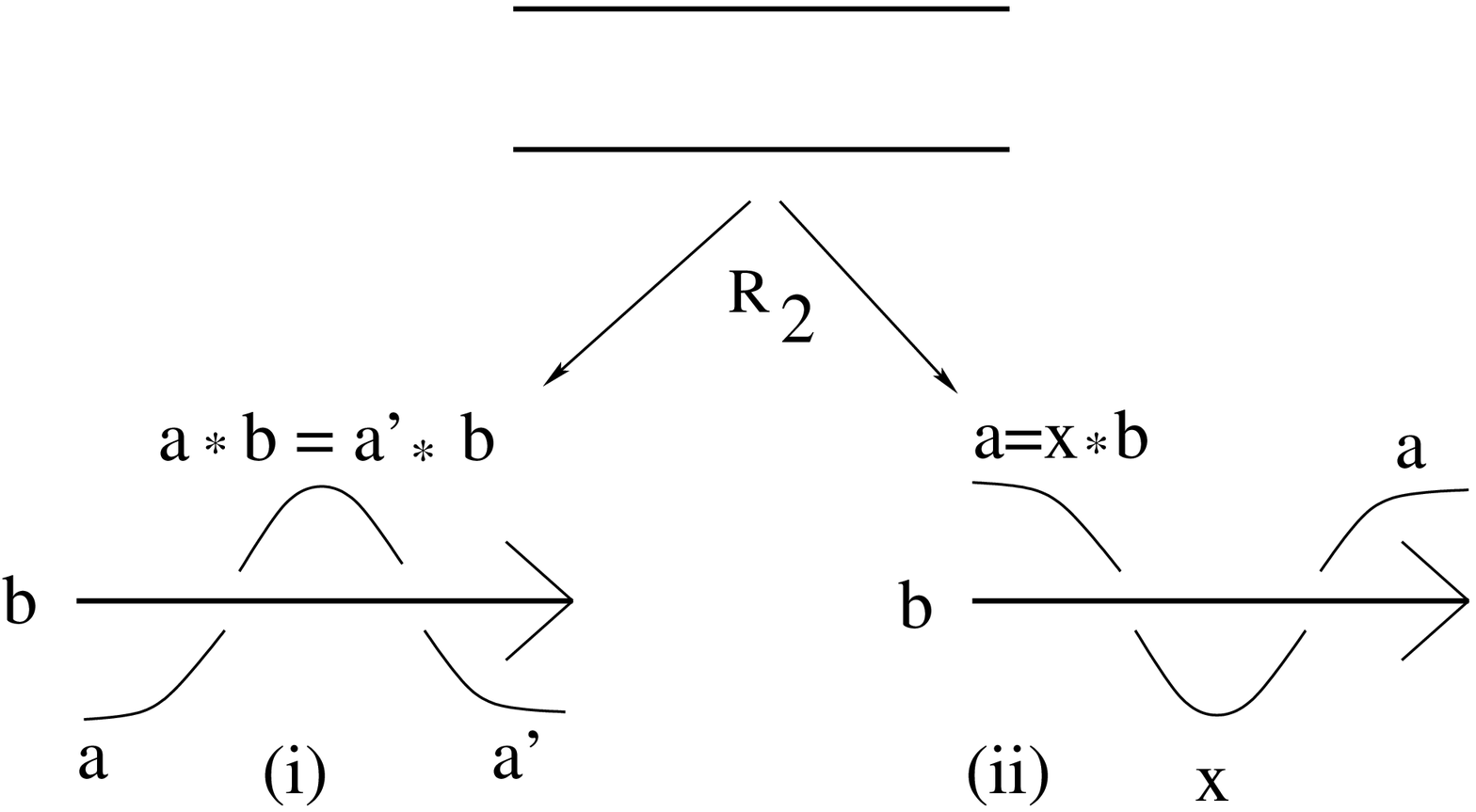,height=6.5cm}}\ \\ \ \\
\centerline{Figure~10.3;  $col_X(R_2(D))= col_X(D)$}
\begin{enumerate}
\item[$(R_3)$] We illustrate the need for right self-distributivity of $*$ in Figure 10.4, where we choose 
all crossings to be positive. If $*$ is also invertible, then all other choices of orientation 
 follow as well (Proposition \ref{Proposition 2.4} can be used then).
\\ \ \\
\centerline{\psfig{figure=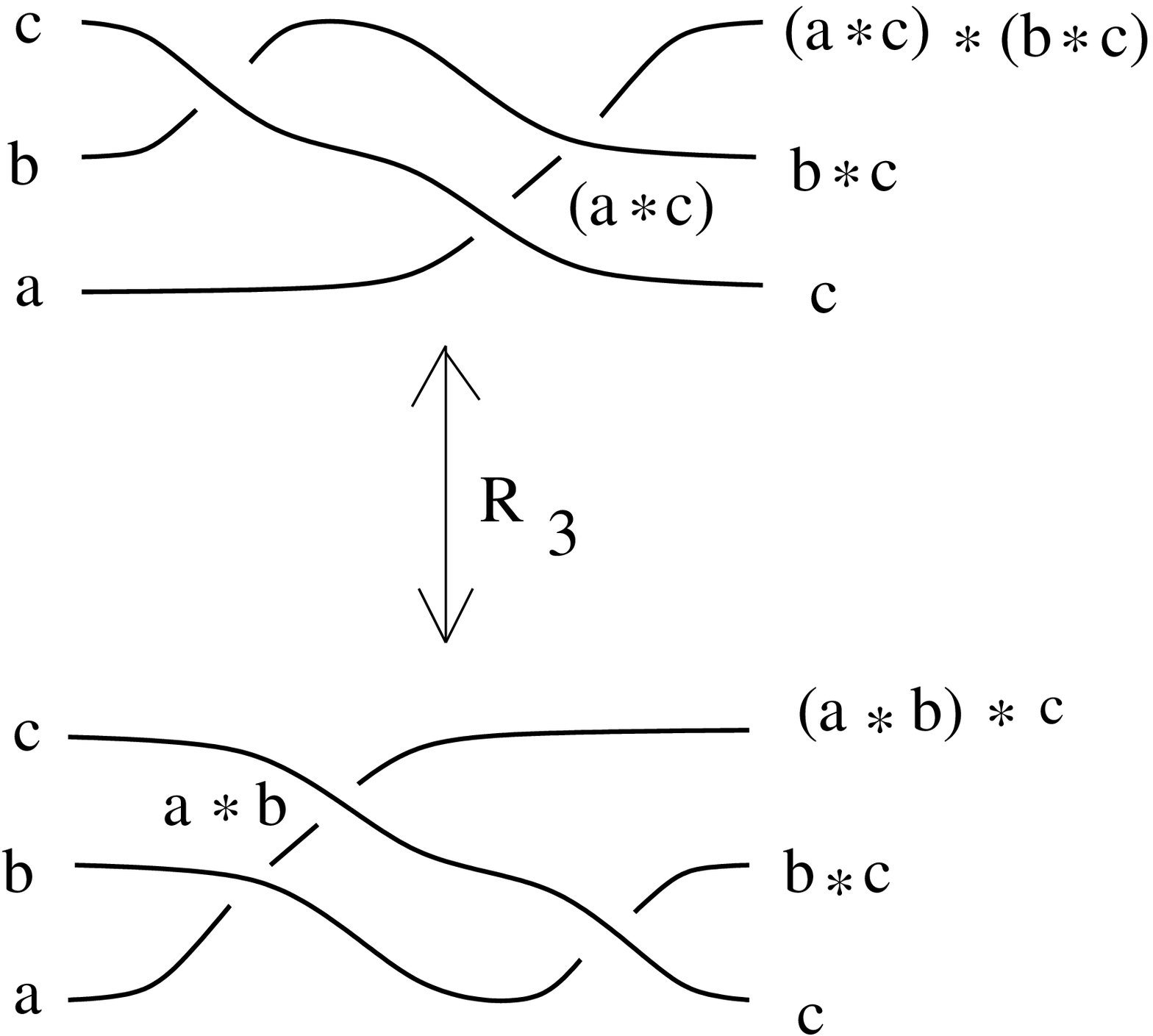,height=5.1cm}}
\\ \ \\
\centerline{Figure 10.4;  $col_X(R_3(D))= col_X(D)$}
\end{enumerate} 

\subsection{ Motivation for degenerate chains and quandle homology}\label{Subsection 10.1}\

Carter, Kamada and Saito at al. noticed in 1998 \cite{CKS,CJKLS} that if one colors a link diagram $D$ 
by elements of a given
quandle $(Q,*)$ as described above, and then considers a sum over  all crossings of $D$ 
of pairs in $Q^2$, $\pm (a,b)$, according to the following convention:
\\ \ \\
\centerline{\psfig{figure=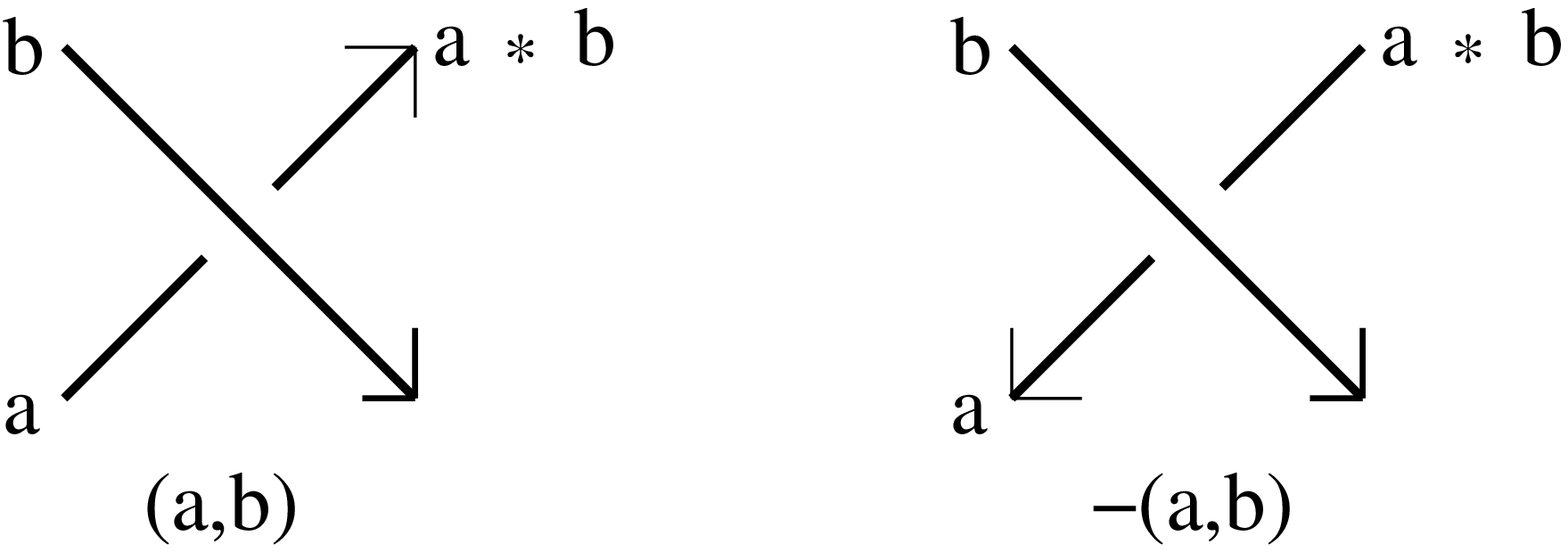,height=3.4cm}} 
\ \\ 
\centerline{Figure 10.5;  building the 1-chain for an oriented link diagram and its coloring}
\ \\ 
then the sum $c(D)=\sum_{v \in \{Crossings\}} sgn(v)(a(v),b(v))$ is not only a 1-chain, but is a 1-cycle 
in $C_1(Q)$, and its class in the first homology $H^Q_1(Q)$ is invariant under Reidemeister moves.
We show this carefully, and in particular, stress the difference (shift) in grading. The history 
of discovering quandle homology is surveyed in \cite{Car}.

\begin{enumerate}
\item[(i)] Carter, Kamada, and Saito have considered cocycle invariants, and in their convention, an element $Q^2 \to \Z $ 
is a 2-cocycle. For us, however, the sum constructed above is a 1-cycle, an element of 
$C_1(Q)=\Z Q^2$. The (rack or quandle) boundary operation they consider is $\partial^R= \partial^{(*)}- \partial^{(*_0)}$, 
and for $\Z Q^2$ we get $\partial^R(x_0,x_1)= x_0 - x_0*x_1$. In our case, if a contribution of the crossing $v$ 
is $sgn(v) (a,b)$, then its contribution to $\partial^R(c(D)$ is $sgn(v) (x_0 - x_0*x_1)$. Figure 10.5 
informs us that the contribution is exactly the difference between the label of an undercrossing  at the entrance 
minus the label at the exit (when moving according to the orientation). Thus, obviously, each component 
contributes zero to $\partial^R(c(D))$, thus $\partial^R(c(D))$ is a 1-cycle.
\item[(ii)] The first Reidemeister move introduces $(a,a)$ into the sum, so we have to declare it to 
be zero; here the need to consider normalized or quandle homology arises. Now $\partial^Q: C/C^D \to C/C^D$.
The second Reidemeister move always works as the crossings involved in it have opposite signs, so 
the new contribution to $w(c)$ cancels.
\item[(iii)] With the third Reidemeister move, we consider the move from Figure 10.4 (it requires some 
topological manipulation, but it is well known that it is sufficient). Thus 
the contributions to $c(D)$ of three crossings from the top diagram is $(b,c)+(a,c)+(a*c,b*c)$ and 
of the bottom diagram is $(a,b)+ (a*b,c) + (b,c)$. Now $\partial^R(a,b,c)= (a,c)-(a,b)-(a*b,c)+(a*c,b*c)$,
which is exactly $c(D)-c(R_3(D))$. Thus $c(D)$ and $c(R_3(D))$ are homologous in $H_1^R(D)$ and $H^Q_1(D)$.
\item[(iv)] We showed that if $(Q,*)$ is any quandle and we choose a quandle coloring of $D$, then $c(D)$ 
yields an element of $H_1^Q(Q)$ preserved by all Reidemeister moves. However, if $(Q,*)$ is only a rack, 
then $c(D)$ is an invariant of $R_2$ and $R_3$ (a so called invariant of regular isotopy), thus 
$c(D)$ yields an element of $H_1^Q(Q)$ invariant up to regular isotopy.
\item[(v)] One can improve (iv) slightly and make our cycle invariant $c(D)$ more useful  by noting that 
$c(D)$ yields an invariant of framed isotopy. Here we observe that we can move a ``kink" of the first Reidemeister 
move under another arc using $R_2$ and $R_3$ only, and cancel contributions from ``kinks" of the opposite sign, 
as long as they are in the same component.
\end{enumerate}
The above considerations have been generalized to surfaces in 4-space, or more generally, to 
codimension two embeddings; in fact, it was an initial motivation for Fenn, Rourke, and Sanderson to 
introduce rack homology around 1990.  
There is another remarkable cocycle invariant developed in \cite{R-S,CKS} for codimension 2 embeddings, 
 coming from shadow colorings by elements of $(X;*)$. It is a 3-cocycle invariant in classical knot theory 
(we formulate it below in a homology language and with a dimension shift; thus we construct a 2-cycle in $C_n(X)$).
\begin{definition}\cite{R-S,CKS}\label{Definition 10.1}
Let $(X,*)$ be a rack and $D$ an oriented link diagram. We decorate arcs of $D$  by elements of $X$ as in 
the previous definition (Figure 10.1). Additionally, we color regions of $R^2 -D$ by elements of $X$ according 
to the convention: \parbox{2.5cm}{\psfig{figure=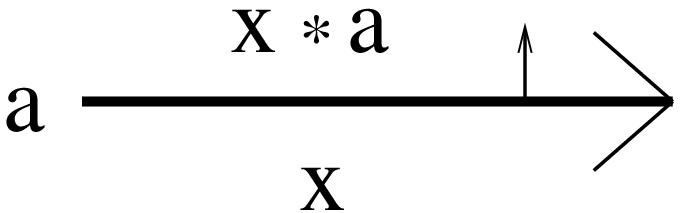,height=0.8cm}} (the small arrow is added to 
record a positive orientation of the projection surface).
For a given shadow coloring we define a 2-cycle $c_2(D)\in C_2^R(X)$ as the sum over all crossings of $D$ 
of terms $\pm (x,a,b)$ according to the convention of Figure 10.6: \\ \ \\
\centerline{\psfig{figure=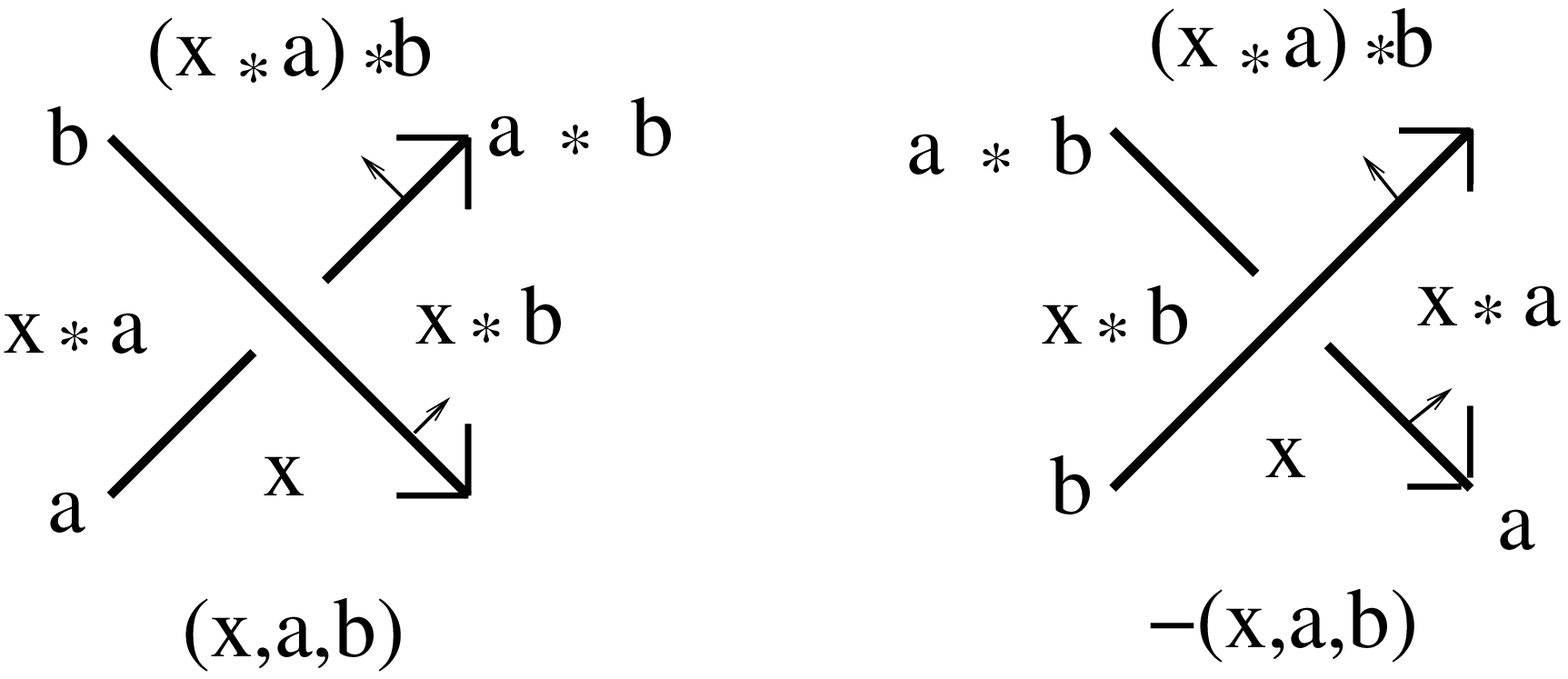,height=3.8cm}}
\ \\
\centerline{Figure 10.6;  building a 2-chain for an oriented link diagram using a shadow coloring}
\ \\
One can check that $c_2(D)$ is a 2-cycle in $C_2(X)$. Further,
$c_2(D)$ is preserved by the second Reidemeister move (to see the cancellation of contributions 
from two new crossings after $R_2$, we should just put together crossings of Figure 10.6). 
With a little more effort one shows that $c_2(R_3(D)) - c_2(D)$ is a boundary (e.g. if we shade 
regions of Figure 10.4, with the bottom region labelled by $x$, then $c_2(R_3(D)) - c_2(D) = \partial (x,a,b,c)$). 
Thus $c_2(D)$ and $c_2(R_3(D)$ 
are homologous in $H_2^R(D)$. To summarize, the homology class of $c_2(D)$ is a regular isotopy invariant.\\
If $(X,*)$ is a quandle, we can work with quandle homology $H^Q_2(X)$, and because the contribution of the 
new crossing in a first Reidemeister move is a degenerate element, the class of $c_2(D)$ in 
$H^Q_2(X)$ is preserved by all Reidemeister moves.
\end{definition}

If we only care about the third Reidemeister move of Figure 10.4, we can work with any shelf $(X,*)$.  
The usefulness of working only with some Reidemeister moves may be debated, but there is already 
a considerable body of literature on the topic \cite{CESS}. 

\begin{remark}\label{Remark 10.2} Recall that the map $p_0: C_n(X) \to C_{n-1}(X)$ is given by\\
 $p_0(x_0,x_1,...,x_n) = 
(x_1,...,x_n)$ and that, as noted in Lemma \ref{Lemma 8.7}, $(-1)^{n+1}p_0$ is a chain map on 
$(C_n\otimes \Z_{\Sigma}, \partial_n^{(a_1,...,a_k)})$.
If $\Sigma =\sum_{i=1}^ka_i =0$, as is the case for rack homology, then $(-1)^{n+1}p_0$ is a chain map.
Our observation is that $p_0(c_2(D))= c(D)$, which follows from the construction, but should have 
some interesting consequences. It is true, in general, that for a given $n$-dimensional ``diagram" $D$ 
of an $n$-dimensional manifold in $R^{n+1}$, the $n$-chain corresponding to a shadow coloring of $D$ is 
sent by $p_0$ to a coloring of $D$. We plan to address the significance of this in \cite{P-R}.
\end{remark}

\section{Yang-Baxter Homology?}
\subsection{From self-distributivity to Yang Baxter equation}

Let $(X;*)$ be a shelf and $kX$ a free module over a commutative ring $k$ with basis $X$ (we can
call $kX$ a {\it linear shelf}). Let $V=kX$, then $V\otimes V = k(V^2)$ and the operation $*$ yields
a linear map $R=R_{(X;*)}: V\otimes V \to V\otimes V$ given by $R(a,b)=(b,a*b)$. Right self-distributivity
of $*$ gives the equation of linear maps $V\otimes V \otimes V \to V\otimes V\otimes V$:
$$ (R\otimes Id)(Id \otimes R)(R\otimes Id) = (Id \otimes R)(R\otimes Id)(Id \otimes R).$$
In general, the equation of type (1) is called a Yang-Baxter equation and the map $R$ a Yang-Baxter operator.
We also often require that $R$ is invertible. With relation to this, we notice that if $*$ is invertible
then $R_{(X;*)}$ is invertible with $R^{-1}_{(X;*)}(a,b)= (b\bar * a, a)$.

In our case $R_{(X;*)}$ permutes the base $X\times X$ of $V\otimes V$, so it is called a permutation or
a set theoretical Yang-Baxter operator. Our distributive  homology, in particular our rack homology
 $(C_n,\partial^R=\partial^{(*)}-\partial^{(*_0)})$, can be thought of as the homology of $R$.
It was generalized from the Yang-Baxter operator coming from a self-distributive $*$ to any
permutational Yang-Baxter operator (coming from biracks or biquandles), \cite{CES-2}. For a general
Yang-Baxter operator, there is no general homology theory (compare \cite{Eis-1,Eis-2}).
The goal/hope is to define homology for any Yang-Baxter operator, so that the Yang-Baxter operator
defining the Jones polynomial leads to a version of Khovanov homology.

\section{Acknowledgements}
I was partially supported by the  NSA-AMS 091111 grant, 
by the Polish Scientific Grant: Nr. N-N201387034, and by the GWU REF grant.

I would like to thank participants of my lectures, readers of early versions of this paper, and the referee for many 
useful comments and suggestions.

\ \\
Address:\\
Department of Mathematics,\\
George Washington University,\\
przytyck@gwu.edu \\
and Gda\'nsk University
\end{document}